\documentclass{article}
\usepackage{babel}
\usepackage[T1]{fontenc}
\usepackage[utf8]{inputenc}
\usepackage{lmodern}
\usepackage[a4paper]{geometry}
\usepackage[all]{xy}
\usepackage{amsmath}
\usepackage{amssymb}
\usepackage{amsfonts}
\usepackage{amsthm}
\usepackage{shorttoc}
\usepackage{bbm}
\usepackage{mathcomp}
\usepackage{amssymb}
\usepackage{relsize}
\usepackage{graphicx}
\usepackage{amsthm}
\usepackage{caption}
\usepackage{enumerate, color}
\usepackage{hyperref}
\usepackage{tikz-cd} 
\usepackage{xfrac}
\usepackage{svg}
\usepackage{faktor}

\usepackage{geometry}
\newcommand{\DD}{\mathbb{D}}

\newcommand{\NN}{\mathbb{N}}

\newcommand{\QQ}{\mathbb{Q}}
\newcommand{\RR}{\mathbb{R}}
\newcommand{\ZZ}{\mathbb{Z}}
\renewcommand{\SS}{\mathbb{S}}

\newcommand{\Bb}{\mathcal{B}}

\newcommand{\St}[1]{\|#1\|_X}
\newcommand{\StS}[1]{\|#1\|_\Sigma}
\newcommand{\Sts}[1]{\|#1\|_M}

\newtheorem{thm}{Theorem}[section]
\newtheorem{cor}{Corollary}[thm]
\newtheorem{lemma}[thm]{Lemma}
\newtheorem{prop}[thm]{Proposition}
\theoremstyle{definition}
\newtheorem{definition}[thm]{Definition}

\begin{document}
	
	\title{\bf On the stable norm of slit tori and the Farey sequence}
	\author{\sc Pablo Montealegre}
	\date{}
	
	\maketitle
	
	\begin{abstract}
		Let $(M,g)$ be a compact manifold endowed with a possibly singular Riemannian metric. The metric induces a norm on the homology of $M$, called the stable norm. We provide explicit computations of the stable norm of flat slit tori using the Farey sequence. We then glue several slit tori together to produce half-translation surfaces whose unit ball of the stable norm has faces of maximal dimension. Furthermore, we give a sub-quadratic estimate for the asymptotic counting of simple homology classes on these surfaces.
	\end{abstract}

\section*{Introduction}

Let $(M,g)$ be a connected compact manifold endowed with a smooth Riemannian metric $g$. The metric induces a norm on the first homology group $H_1(M,\RR)$, named the stable norm by Gromov in \cite{Gromov}. More precisely, following Massart in \cite{Mas}, let $h\in H_1(M,\ZZ)$ be an integral homology class. We define
	$$f(h) = \inf\limits_{\gamma} l(\gamma)$$
	where the infimum is taken over all the closed rectifiable curves $\gamma$ on $M$ representing the class $h$, and where $l(\gamma)$ denotes the length of $\gamma$. Proposition 1.1.3 of \cite{Mas} yields that the limit
	$$ \lim\limits_{N\to \infty} \frac{f(Nh)}{N} $$ 
	exists for any integral class $h$; we denote this limit $\Sts{h}$. One can check that for any two integral homology classes $h_1$ and $h_2$, the triangle inequality $\Sts{h_1+h_2} \le \Sts{h_1}+\Sts{h_2}$ holds. Moreover, for any $N\in\ZZ$ we have $\Sts{Nh}=|N|\Sts{h}$ so by homogeneity we can extend the function $h\mapsto \Sts{h}$ to the rational homology classes $H_1(M,\QQ)$. Finally, by a density argument, we can extend this function to the whole real homology $H_1(M,\RR)$, and proposition 1.1.5 of \cite{Mas} tells us that the extension is a norm. \newline \newline

To this day, little is known about the stable norm of manifolds. In particular, there are very few known explicit examples: see Babenko \cite{babenko} on $2$-tori, Burago-Ivanov-Kleiner \cite{Bur-Iva} for $n$-tori or McShane-Rivin \cite{McShane Rivin} on the punctured hyperbolic torus. 
 
	  Since the stable norm is defined only using length of curves, it still makes sense to define a stable norm on surfaces endowed with singular Riemannian metrics, such as translation and half-translation surfaces. These surfaces enjoy very nice geometrical properties: in particular, they are endowed with a flat metric with a finite number of conical singularities (see for instance Zorich's survey \cite{Zorich}). This greatly simplifies the task of finding geodesics and computing their length, so one can hope to be able to compute explicitly the stable norm on some of these surfaces. Unfortunately, it is still a difficult problem, so the idea of this paper is to compute the stable norm on flat slit tori that we can later on glue together to obtain translation and half-translation surfaces. This is not new: since slit tori are the building blocks of surfaces of main interest in modern geometry and dynamical systems, they are often the testing ground when trying to understand a new object, see for example \cite{Schmoll} or \cite{Sanchez}. \newline

Let us consider the flat square torus with a vertical slit, denoted $X$, that we will call for short the slit torus. More precisely, take a square of area $1$, and cut it along a vertical open interval $I$ of length $\rho$, with $0<\rho<1$. Now glue the opposite edges of the square by translation and compactify it by adding a boundary along each side of the slit, so that the cut is homeomorphic to a circle. The resulting surface is our slit torus $X$, which is homeomorphic to a compact surface of genus $1$ with one boundary component. The slit torus is endowed with a Euclidean metric, that on the interior of the slit torus is simply the restriction of the canonical flat metric on the usual square torus.
	
	\begin{center}
		\includegraphics[scale=0.8]{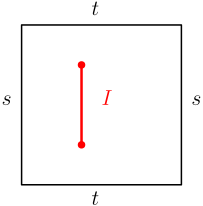}
		\captionof{figure}{Polygonal model of the slit torus}
        \label{modele}
	\end{center}

	Since the slit torus is homotopy equivalent to the wedge sum of two circles, we have
	
	$$\pi_1(X) \cong F_2$$ where $F_2=F(s,t)=\langle s,t \rangle$ is the free group on the two generators $s$ and $t$. We also have
	$$H_1(X,\ZZ) \cong \ZZ^2$$ From now on we will identify the elements $h$ of $H_1(X,\ZZ)$ with the integer couples $(m,n)$ in $\ZZ^2$. \newline

 Our goal is to compute explicitly the stable norm of any given homology class on $X$, or equivalently describe the geometry of the unit ball of the stable norm, which is a convex set in $\RR^2$. It turns out that the computation involves the Farey sequence, a famous number sequence.

\begin{definition}
	The Farey sequence of order $k\in\NN^*$ is the ordered sequence $F_k$ of irreducible fractions in $\QQ$ whose denominator is less than or equal to $k$.
\end{definition}  

For example $F_1= \{... -1/1, 0/1,1/1,2/1,...  \}$, $F_2 =  \{ ..., -1/1, -1/2, 0/1, 1/2,1/1,3/2,2/1,...  \}$. The Farey sequence is omnipresent in mathematics, mainly due to its remarkable combinatorial properties. We sum up some of those in the following propositions, that can be found in \cite{Har}. 

\begin{prop}
	Let $a/b< c/d$ be two irreducible fractions that are consecutive in some Farey sequence $F_k$ of order $k$. These fractions are called a Farey pair and are said to be Farey neighbours, and they satisfy the following properties:
	\begin{enumerate}
		\item bc-ad=1. Note that this property is actually equivalent to being Farey neighbours.
		\item If $p/q$ is an irreducible fraction that is neighbour to both $a/b$ and $c/d$ in some Farey sequence (of order greater than $k$), then
		$$ \frac{p}{q} = \frac{a+c}{b+d}$$
		We denote $p/q = a/b \oplus c/d$ and we say that $p/q$ is the mediant of $a/b$ and $c/d$.
	\end{enumerate}
\end{prop}

Thus if $a/b$ and $c/d$ are Farey neighbours, the first fraction $p/q$ that appears between them as the order of the sequence increases is their mediant: we say that $p/q$ is the \emph{Farey child} of $a/b$ and $c/d$, and that $a/b$ and $c/d$ are the \emph{Farey parents} of $p/q$. Now, the mediant $a/b \oplus p/q$ is also a Farey child of $a/b$: the children of $a/b$ are all the rationals one of whose Farey parents is $a/b$. Interestingly, the integers are the Farey ancestors of all the rationals numbers: indeed, starting from the integers $n/1$ and taking the successive mediants one can obtain any rational number.

\begin{prop}
	Let $p/q = a/b \oplus c/d$ be an irreducible fraction, with $a/b<c/d$ its Farey parents. The Farey parents are the best rational approximations of $p/q$ with denominator less than $q$. More precisely, $a/b < p/q < c/d $ and there are no other rational number than $p/q$ with denominator less than $q$ in the interval $[a/b,c/d]$. In other words, the last two convergents in the continued fraction expansion of $p/q$ are its Farey parents.
\end{prop}

We can now state our main result.

\begin{thm}
Let $X$ be the square flat torus with a vertical slit of length $\rho$. Let $h=(m,n)\in H_1(X,\ZZ)$ be a primitive integral homology class, and let $a/b$ and $c/d$ be the Farey parents of $n/m$, with $a/b<c/d$. The unit ball $\Bb$ of the stable norm of $X$ has an vertex in the direction $h$ if and only if either

\begin{itemize}\renewcommand{\labelitemi}{$\bullet$}
\item $|m|\le 1/\rho$, and in that case $\St{(m,n)}=\sqrt{m^2+n^2}$ \\ \\
or
\item $|m|>1/\rho$ and $(|b|\le 1/\rho$ or $|d|\le 1/\rho)$, and in that case $\St{(m,n)}=\sqrt{b^2+(a+\rho)^2}+\sqrt{d^2+(c-\rho)^2}$. 
\end{itemize}
In every other direction of $H_1(X,\RR)$, the point of the unit sphere $\partial B$ of the stable norm in this direction lies in the interior of a segment. In particular, if $m,b,d> 1/\rho$ we have $\St{(m,n)}=\St{(b,a)}+\St{(d,c)}$.
\end{thm}

Note that since the unit ball $\Bb$ of the stable norm has infinitely many vertices it is not a polygon, although it very much looks like one at first glance. Sections 1, 2 and 3 are devoted to proving this theorem. In section 4 we extend our result to general slit tori, i.e to slit tori whose flat metric comes from any parallelogram instead of from a square, and with a slit of any direction. \newline 

A question one might ask is which convex bodies can be realized as the unit ball of the stable norm of a manifold. In dimension $n\ge 3$, Babenko and Balacheff showed that given a closed smooth manifold $M$, any centrally symmetric polytope with vertices of rational direction can be realized as the unit ball of the stable norm of a Riemannian metric on $M$.
However, in dimension $2$ Massart showed in \cite{Mas2} that the dimension of the flats of the unit ball of a closed orientable \emph{smooth} surface of genus $g$ is at most $g-1$; in particular, the unit ball of the stable norm is far from being a polytope. But what about the stable norm of singular metrics?
In section 5, we glue two slit tori $X_1$ and $X_2$ along a long flat cylinder to obtain a genus $2$ half-translation surface $\Sigma$ on which we are able to compute the stable norm.

\begin{thm}\label{convex hull}
    The unit ball of the stable norm of $\Sigma$ is the convex hull of the set $$ \left \{ \frac{(m,n,0,0)}{\StS{(m,n,0,0}} \text{ with } (m,n)\in V(X_1)  \right \} \cup \left\{\frac{(0,0,p,q)}{\StS{(0,0,p,q)}} \text{ with } (p,q)\in V(X_2)\right \} $$ where $V(X_i)$ is the set of directions of vertices of the stable norm of $X_i$.
\end{thm}

In particular, the unit ball of the stable norm of $\Sigma$ has faces of dimension $3$, which is maximal since $H_1(\Sigma,\RR)$ has dimension $4$. More generally, glueing $n$ slit tori with the same construction gives us a half-translation surface of genus $n$ whose flats of the stable norm are of maximal dimension, that is $2n-1$. Again, the unit ball of the stable norm of $\Sigma$ is not a polytope as it has infinitely many vertices, but it looks a lot like one. Although getting such a bizarre-looking unit ball as we abandon the smoothness assumption on the metric is not a big surprise, it is remarkable that it happens after losing so little smoothness, as one could argue that half-translation surfaces are as nice as singular surfaces can get.
\newline 

Finally, in section 6 we investigate a counting problem related to the stable norm. The classical simple curve counting problem on surfaces is the following: how many are there simple closed curves of length less than a positive number $x$ on a given surface? This question, along with many related problems, has seen significant progress during the last two decades. Now, one might ask which homology classes do these curves represent. We say that an integral homology class on a surface $S$ is \emph{simple} if its stable norm is realized by the length of a simple closed curve on $S$. Then the curve counting problem becomes: how many are there simple homology classes of stable norm less than $x$ on a given surface? It is known (see Balacheff-Massart \cite{Bal-Mas} Theorem C) that it is possible to find closed non-orientable surfaces with only a finite number of simple homology classes. On the other hand, Babenko \cite{babenko} in dimension $2$ and later Burago-Ivanov-Kleiner \cite{Bur-Iva} in every dimension showed that on flat tori there are infinitely many simple homology classes, as in the case of the punctured hyperbolic torus (McShane-Rivin \cite{McShane Rivin}). Thus, can we find asymptotic estimates? Gutkin and Massart \cite{Gut-Mas} showed, based on previous work by Masur \cite{Masur88}, \cite{Masur90} on the growth rate of closed orbits for quadratic differentials, that the number of simple homology classes of stable norm less than $x$ on translation surfaces grows quadratically in $x$. For instance on flat tori the simple homology classes correspond to the primitive elements of $\ZZ^2$, so the asymptotic behaviour is that of  $\mathlarger{\frac{6}{\pi}} x^2$. Since any half-translation surface is doubly covered by a translation surface, these two classes of surfaces share many geometrical properties. So, is the asymptotic growth of simple homology classes on half-translation surfaces quadratic? We use the example of the half-translation surface $\Sigma$ constructed in section 5 to show that the answer to this question is negative. 

\begin{thm}\label{comptage}
Let $p(x)$ be the number of simple homology classes on $\Sigma$ with stable norm less than $x$.
Then $$p(x) \sim 8 \left(\sum_{b=1}^{\lfloor 1/\rho \rfloor} \frac{\varphi(b)}{b}\right )  x \ln x + O(x)$$
where $\varphi$ is Euler's totient function.
\end{thm}

We would like to point out that this result agrees with Conjecture 1.2 in \cite{Mas-Par}. Roughly speaking, this conjecture states that for any closed orientable surface $M$ of genus $2$ equipped with a reasonable Finsler metric, the number $p(x)$ of simple homology classes of $M$ with stable norm less than $x$ satisfies
$$ \lim_{x\rightarrow \infty } \frac{p(x)}{x^2} = \Omega$$
where $\Omega$ is the area of the closure of the set of vertices of the unit ball of the stable norm of $M$. In the case of our surface $\Sigma$, the closure of the set of vertices of the unit ball has measure zero, and Theorem \ref{comptage} confirms that $p(x)$ is indeed sub-quadratic. \\

\textbf{Acknowledgments} The author would like to express his gratitude towards Daniel Massart for the many enlightening conversations that led to the present paper, along with his careful reading of several versions of the manuscript. Thanks also to Tom Ferragut and Iv\'an Rasskin for the time and interest they gave kindly to this work.

\section{Visible homology classes}

In this first section we reformulate the problem of computing the stable norm in a more practical way, and we compute the stable norm of a family of homology classes that we call \emph{visible classes}. \newline \newline

The main goal of this paper is to compute explicitely the stable norm of any given homology class on the slit torus $X$. Equivalently, as a norm is uniquely determined by its unit ball, we want to describe the geometry of the unit ball of the stable norm. Note that we only need to compute the stable norm of integral homology classes, as the stable norm of the other classes is obtained by homogeneity and continuity of the norm. So, for any integral homology class $h$, we need to compute $f(Nh)$ with $N\in\NN^*$, where $f(h) = \inf\limits_{\gamma}  \{ l(\gamma), $ with $\gamma$ a closed curve such that $[\gamma]=h \}$. So we need to compute the infimum of the lengths of curves representing $N.h$.  But if $\gamma$ is a closed curve representing a given integral class, so is any curve obtained by translating $\gamma$ upwards, provided that it does not encounter the slit, and the length of the translated curve is the same as $\gamma$'s. So by translating upwards until we encounter the slit any closed curve representing an integral class $h$ is freely homotopic to a closed curve of same length passing through the lower endpoint of the slit. Hence, without loss of generality, we can narrow down our search for short curves representing a given integral class to closed curves passing through the lower endpoint of the slit. \newline

Consider $p:\tilde{X}\longrightarrow  X$ the abelian covering of $X$, i.e the normal covering of $X$ whose transformation group is isomorphic to $H_1(X,\ZZ) \cong \ZZ^2$. One convenient way to picture $\tilde{X}$ is the Euclidean plane from which we remove all $\ZZ^2$-translates of the vertical segment $[(0,0),(0,\rho)]$ of length $\rho$ : this way, the local isometry between $\tilde{X}$ and $X$ is the obvious projection. These vertical segments are the lifts of the slit to $\tilde{X}$: by extension we will also call them the slit. Now, given an integral homology class $h=(m,n)\in\ZZ^2$ on $X$, any closed curve on $X$ representing $h$ (and passing through the lower endpoint of slit, which, as we have seen, we can always assume without loss of generality) can be lifted in a unique way to $\tilde{X}$ as a path  of same length from the origin $(0,0)$ to the point $(m,n)$. Thus, computing $f(h)$ is equivalent to finding the shortest rectifiable path from $(0,0)$ to $(m,n)$ in $\tilde{X}$. More precisely, if we set
	
	$$ \tilde{f}(h) := \inf\limits_{\tilde{\gamma}} l(\tilde{\gamma})$$
	where the infimum is taken over all the rectifiable paths in $\tilde{X}$ from $(0,0)$ to $(m,n)$, we have $f(h) = \tilde{f} (h)$ so $\St{h}= \lim\limits_{N\to \infty} \tilde{f}(Nh)/N$.
	
	\begin{prop}
		Let $h\in H_1(X,\ZZ)$ be an integral homology class on $X$. Then the sequence $(f(Nh)/N)_{N\in\NN}$ is constant.
	\end{prop}

\begin{proof}
	We reproduce an old argument from Hedlund \cite{Hed}, originally stated on the torus. Let $h=(m,n)$ be an integral primitive homology class on $X$, i.e $\gcd(m,n)=1$. Let $N\in \NN^*$ and let $\gamma$ be a rectifiable path on $\tilde{X}$ from $(0,0)$ to $(N.m,N.n)$. We will show that $\gamma$ projects onto $X$ as the union of $N$ closed curves representing $h$, thus $f(Nh) \ge Nf(h)$, and since the converse inequality is trivial we will finally get $f(Nh) = Nf(h)$. To do this, we apply the following algorithm.\newline

	Let $\Delta$ be the line from the origin to $(m,n)$ in $\tilde{X}$ (it might encounter the slit) and let $\Delta_l$ (resp. $\Delta_r$) the leftmost (resp. rightmost) parallel to $\Delta$ which meets $\gamma$. If $\Delta_l = \Delta = \Delta_r$ then $\gamma \subset \Delta$ and there is nothing left to prove, so let us assume that $\Delta_l \ne \Delta_r$. Let $U \in \Delta_l \cap \gamma$, and let $U' = T(m,n).U$, where $T(m,n)$ is the translation by the $(m,n)$ vector. \\ \\

	If $U' \in \gamma$ then we are done: since $U$ and $U'$ project onto the same point in $X$, the arc of $\gamma$ from $U$ to $U'$ projects onto a closed curve in $X$ representing $h=(m,n)$. We then remove this arc from $\gamma$ to obtain two subpaths of $\gamma$, and by concatenation we obtain a new (and shorter) path from $(0,0)$ to $((N-1).m, (N-1).n)$. \\ \\
	
	If $U'$ does not belong to $\gamma$, then $U'$ lies on the boundary of some connected region $\delta$, bounded by $\Delta_l$ and $\gamma$. Let $V \in \Delta_r \cap \gamma$ and assume that $V'=T(m,n).V$ does not belong to $\gamma$ (just as before, if $V' \in \gamma$ then we are done). By construction, $V'$ lies in the exterior of $\delta$: the only way it could belong to $\delta$ would be if he lied on $\gamma$, which we assumed is not the case. Consider the path $\gamma' = T(m,n) \gamma$: it is well-defined because $\gamma$ does not encounter the slit and the copies of the slit are all related by integral translations. By construction, the path $\gamma'$ passes through $U'$ and $V'$, one in the interior of $\delta$ and the other on its exterior: since $\gamma'$ is continuous, it must cross transversally the boundary of $\delta$ at some point $P'$. Since $\gamma'$ cannot intersect transversally the line $\Delta_l$, the point $P'$ must lie on $\gamma$ instead. Thus, the arc of $\gamma$ from $P := T(-m,-n) P' $ to $P'$ projects onto a closed curve on $C$ representing $h$. We then refer to the previous case: as before, we remove this arc from $\gamma$ to get a shorter path. \\ \\
	
	We repeat this process $N$ times, until we have decomposed $\gamma$ as $N$ arcs, each projecting onto representatives of $h$ in $X$.

 \begin{center}
	\includegraphics[scale=0.6]{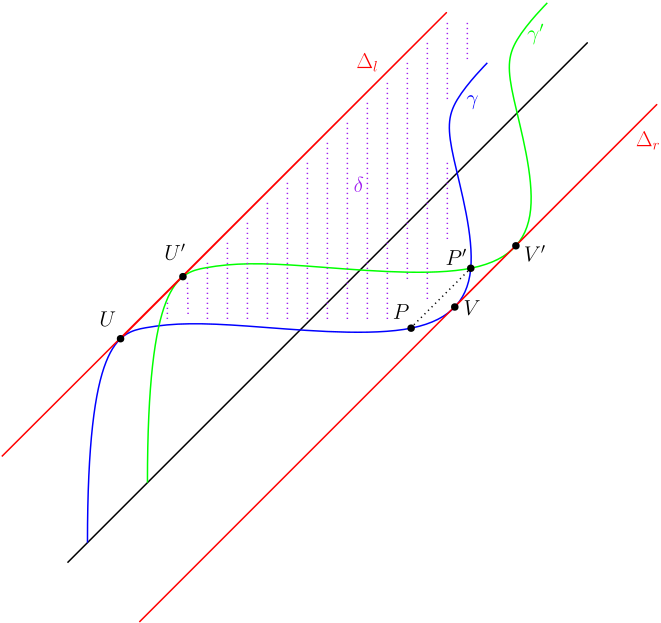}
	\captionof{figure}{The geometric construction from the proof}

\end{center}
	
\end{proof}

Since the set of arc length parametric closed curves in $X$ representing an integral homology class $h$ whose length does not exceed a fixed constant $C>0$ is compact, by an Arzela-Ascoli argument we get
\begin{cor}
	Let $h\in H_1(X,\ZZ)$ be an integral homology class. Then $f(h)= \St{h}$, and the infimum in the definition of $f$ is in fact a minimum. 
\end{cor}

It follows that in order to compute the stable norm of a given integral class $h$, we have to find the length of its shortest representative on $X$. 

\begin{definition}
    Let $h\in H_1(X,\ZZ)$ be an integral homology class, and let $\gamma$ be a closed curve representing $h$ on $X$. If $l(\gamma)=\St{h}$ we say that $\gamma$ is a minimizing curve for $h$, or for short that $\gamma$ is minimizing. Such a minimizing representative might not be unique.
\end{definition}

Finally, note that since a closed curve on $X$ has the same length as its lift to the abelian covering $\tilde{X}$, by transporting the problem to $\tilde{X}$ computing $\St{(m,n)}$ is equivalent to a path-finding problem, that is finding the shortest path from $(0,0)$ to $(m,n)$ in a grid with vertical obstacles of length $\rho$ based at every integer point of the grid. This is commonly known in computer sciences as an \emph{any-angle path planning} problem, on which there already exist many efficient numerical methods (e.g see \cite{info}). \newline

Let $h=(m,n)$ be an integral homology class on $X$ and assume that $h$ is primitive, i.e $gcd(m,n)=1$. We want to compute the length of the shortest representative of $h$ on $X$, or equivalently the length of the shortest path from $(0,0)$ to $(m,n)$ on $\tilde{X}$. Then the natural candidate for being such a path is the segment $[(0,0) , (m,n)]$: if this segment does not intersect the slit, it projects onto a closed curve on $X$ that represents $h$, and this curve is obviously minimizing for $h$. So, does it intersect the slit? \newline 

Since the problem is symmetric, to simplify notations assume that $m,n\ge 0$. 
The segment $[(0,0),(m,n)]$ follows the equation $y=\mathlarger{\frac{n}{m}}x$, with $x\in [0,m]$. It crosses the vertical of the slit at points $(k,k\mathlarger{\frac{n}{m}})$, $k=1,...,m-1$: those are the points where an intersection with the slit may occur. Recall that the slit has length $\rho$.
Since the copies of the slit are the open segments $\left \{ (p,q+s) \text{ with } s\in]0,\rho[ \right \}$ with $(p,q)\in \ZZ^2$, we have to check how the fractional part $\left \{ k\mathlarger{\frac{n}{m}} \right \} $ of $ k\mathlarger{\frac{n}{m}} $ compares with $\rho$ for all $k=1,...,m-1$. More precisely, there is an intersection with the slit at the vertical $x=k$ if and only if

$$\left \{ k\frac{n}{m} \right \}< \rho$$

Remark that $\left \{ k\mathlarger{\frac{n}{m}} \right \} = \mathlarger{\frac{r}{m}}$, where $r$ is the remainder in the Euclidean division of $kn$ by $m$, i.e $kn = r$ in $\ZZ / m\ZZ$. Since $m$ and $n$ are coprime, $n$ is invertible in $\ZZ / m\ZZ$ and multiplication by $n$ is an automorphism of $\ZZ /m\ZZ$, so $r$ takes all the possible values $1,2,...,m-1$ when $k$ varies from $1$ to $m-1$. Hence, the smallest possible value for $\left \{ k\mathlarger{\frac{n}{m}} \right \}$ is $\mathlarger{\frac{1}{m}}$. Now, 
$$\frac{1}{m} < \rho \iff m\rho > 1$$
Hence if $m\rho \le 1$, the segment $[(0,0),(m,n)]$ does not encounter the slit and projects onto a minimizing representative of $(m,n)$ on $X$ of length $\sqrt{m^2 + n^2}$. We have shown:

\begin{prop}\label{visible}
	Let $h=(m,n)\in H_1(X,\ZZ)$ be a primitive integral homology class. Let $\rho$ be the length of the slit. The stable norm coincides with the Euclidean norm in the direction $(m,n)$ if and only if $m\rho \le 1$.
	Moreover, if $m\rho > 1$ then $\St{(m,n)} > \|(m,n)\|_2$.
\end{prop}

By the previous proposition, if $m\rho\le 1$ then a light ray emitted from $(0,0)$ and travelling in a straight line would illuminate the point $(m,n)$ on $\tilde{X}$. Hence the following terminology, borrowed from \cite{Har}:

\begin{definition}
    A primitive integral homology class $(m,n)$ is visible if $m\rho \le 1$, or equivalently if its stable norm coincides with the Euclidean norm. A visible direction in $H_1(X,\RR)$ is the direction of a visible homology class.
\end{definition}

\begin{cor}
	Let $\Bb$ be the unit ball of the stable norm and let $\DD$ be the unit Euclidean disk in $\RR^2$, with boundary $\SS^1$. Then $\Bb \subset \DD$, and $\partial\Bb \cap \SS^1$ is non empty, with intersections occurring exactly in the visible directions.
\end{cor}

 Remark that since $\rho < 1$ the classes of the form $(1,n)$ are always visible, whatever value $\rho$ may take. This makes sense since the segment from $(0,0)$ to $(1,n)$ never crosses the vertical of the slit, so there is no risk of intersecting the slit. \newline 

The unit ball of the stable norm $\Bb$ is a convex set of the plane, symmetric with respect to the origin. We say that $\Bb$ has \emph{a flat} in a direction $(m,n)$ if the unit vector $\frac{(m,n)}{\St{(m,n)}}$ lies in the interior of a segment contained in the boundary $\partial \Bb$ of the unit ball. If $\Bb$ does not have a flat in the direction $(m,n)$ we say that $\Bb$ is strictly convex in the direction $(m,n)$, and the corresponding point on the boundary of $\Bb$ is called an extreme point. Since the unit circle $\SS^1$ is strictly convex in every direction, we get

\begin{cor}
	The unit ball of the stable norm is strictly convex in every visible direction.
\end{cor}

Now, given a primitive integral homology class $h=(m,n)$ that is not visible, i.e $m\rho> 1$, we know a minimizing representative of $h$ cannot be the projection of the straight line segment from $(0,0)$ to $(m,n)$. To find exactly what a minimizing representative looks like, we will search separately for short simple and non-simple closed curves, as those two cases are very different. Note that when we say simple curves we actually mean \emph{injective} curves, that is to say curves with no self-intersection point, as opposed to the terminology used by some authors that may allow self-intersections as long as they are not transverse.

\section{Shortest non-simple closed curves}

\begin{lemma}
	Let $h\in H_1(X,\ZZ)$. The unit ball of the stable norm $\Bb$ has a flat in the direction $h$ if and only if there exists a minimizing curve for $h$ with at least one self-intersection. Equivalently, the unit ball $\Bb$ is strictly convex in the direction $h$ if and only if every minimizing curve for $h$ is simple.
\end{lemma}

\begin{proof}
	Let $h$ be a primitive integral homology class, and assume that the unit ball of the stable norm $\Bb$ has a flat in the direction $h$. Then there exist $h_1,h_2 \in H_1(X,\RR)$ such that $\frac{h}{\St{h}}$ lies in the interior of the segment $[\frac{h_1}{\St{h_1}}, \frac{h_2}{\St{h_2}}]$, and up to shrinking this segment we can assume that $h$ is the middle of the segment and that $h_1$ and $h_2$ are both integral classes. Thus, one can get a minimizing (multi)curve for $h$ by taking the union of two curves, representing (and minimizing for) respectively $h_1$ and $h_2$. Let $\gamma$ be such a minimizing multicurve. Assume that $\gamma$ has two connected components $\gamma_1$ and $\gamma_2$, representing $h_1$ and $h_2$ respectively. By translating each component upwards until they encounter the lower endpoint of the slit, we get another representative of $h$ that is also minimizing, because applying a translation does not change the homology class nor the length of each curve. This representative is connected and has at least one self-intersection, located at the lower endpoint of the slit. \\ 
	
	Conversely, if $h$ is minimized by a closed curve $\gamma$ with at least one self-intersection at a point $p\in X$, one can see $\gamma$ as the union of two closed curves $\gamma_1$ and $\gamma_2$ based at $p$. These curves necessarily have non-trivial homology classes $h_1$ and $h_2$: otherwise, by simply deleting one of them from $\gamma$ we would get another closed curve representing $h$ that is strictly shorter than $\gamma$, which contradicts its minimality. Moreover, these curves must be minimizing in their homology class, simply because if one of them is not we could replace it by a minimizing curve of its homology class to get another representative of $h$ shorter than $\gamma$, again contradicting the minimality of $\gamma$. Hence $\St{h}=\St{h_1}+\St{h_2}$ and $h = h_1 + h_2$, so $\mathlarger{\frac{h}{\St{h}}}$ lies in the interior of the segment $\left [\mathlarger{\frac{h_1}{\St{h_1}}, \frac{h_2}{\St{h_2}}}\right ]$ and the unit ball of the stable norm has a flat in the direction $h$.
\end{proof}

This will be useful because, given an integral homology class $h$, using the Farey sequence it is rather easy to find an expression for the shortest non-simple closed curve representing $h$ whose connected components are not homologically trivial. Thus if $\Bb$ has a flat in the direction $h$, since $h$ must be minimized by a non-simple curve, we will get an expression for its stable norm. The idea of the link with the Farey sequence is the following. Given a non-visible homology class $(m,n)$, we want to find the shortest path from the point $(0,0)$ to the point $(m,n)$ in $\tilde{X}$, but since the class is not visible the straight line path is forbidden as it eventually intersects the slit. Hence we have to deviate from this direction at some point. Since the straight line segment would be the shortest path to the point $(m,n)$ if it did not intersect the slit, we want to deviate from it as little as possible in order to minimize the length. In which directions should we deviate? One can see a direction as a slope, and notice that the slope of the line from the origin to the point $(b+d,a+c)$ is the rational $(a+c)/(b+d) = a/b \oplus c/d$, which is, assuming that $a/b$ and $c/d$ are a Farey pair, the Farey child of the slopes of the lines from the origin to the points $(b,a)$ and $(d,c)$. To put it in an informal way, "the sum of two directions is their Farey child"; or more generally "the sum of two directions is their mediant". Thus, if $a/b$ and $c/d$ are the Farey parents of $n/m$, the previous observation suggests to consider the classes $(b,a)$ and $(d,c)$, that correspond to the slopes $a/b$ and $c/d$, as they give the best rational approximations of the direction $(m,n)$.

\begin{prop}\label{non simple}
	Let $h=(m,n)$ be a primitive integral homology class with $m\rho  > 1 $, i.e $h$ is not visible. Let $a/b$ and $c/d$ be the Farey parents of the rational $n/m$. Then the shortest non-simple closed curve $\gamma = \gamma_1 \cup \gamma_2$   representing $h$, with $\gamma_1$ and $\gamma_2$ being two closed curves representing non-trivial homology classes, has length $ \St{(b,a)} + \St{(d,c)} $.
	
	In particular, if the unit ball of the stable norm has a flat in the direction $(m,n)$ then the points $\mathlarger{\frac{(b,a)}{\St{(b,a)}}}$ and $\mathlarger{\frac{(d,c)}{\St{(d,c)}}}$ belong to the same flat (they may be its endpoints).
\end{prop}

\begin{proof}
	
	Let $\gamma$ be a non-simple closed curve representing $h=(m,n)$. Since $\gamma$ has a self-intersection we can see $\gamma$ as the union of two closed curves $\gamma_1$ and $\gamma_2$ representing two integral homology classes, respectively $(p,q)$ and $(m-p, n-q)$ with $p,q\in \ZZ$, and we assume that $p$ and $q$ are not both zero.  Recall that $\gamma$ lifts to a path from $(0,0)$ to $(m,n)$ in the abelian covering $\tilde{X}$ of the slit torus, which is the concatenation of a path from $(0,0)$ to $(p,q)$ and of another path from $(p,q)$ to $(m,n)$. Since we are looking for \emph{short} curves representing $h=(m,n)$, we can assume $0<p<m$, $0<q<n$. Indeed, if it is not the case, then $p <0$ or $q<0$ and one the the paths composing the lift of $\gamma$ goes backwards at some point. This is not optimal as it wastes length, so it cannot be a \emph{short} non-simple closed curve representing $h$, let alone the shortest one.
	
	Now, if the homology classes of $\gamma_1$ and $\gamma_2$ are fixed (i.e if $(p,q)$ is fixed), the length of $\gamma$ is minimal if both the lengths of these two curves are minimal in their homology class, so we get the following inequality:
	
	$$l(\gamma) \ge \St{(p,q)}+\St{(m-p,n-q)} $$
	
	with equality if and only if $\gamma$ is the union of two minimizing curves for $(p,q)$ and $(m-p,n-q)$ respectively. Hence, the minimal length of a non-simple closed curve representing $h$ is
	
	$$ \min\limits_{(p,q)} \St{(p,q)}+\St{(m-p,n-q)} $$
	
	with $0<p<m$ and $0<q<n$. Let $C$ be the convex hull of $\left \{ \frac{(p,q)}{\St{(p,q)}}, 0<p<m \text{ and } 0<q<n \right \}$. Since $\Bb$ is convex, it contains the convex hull of any of its subsets, so $C\subset \Bb$. Thus, since $(m,n)/\St{(m,n)} \in \partial\Bb$, either $(m,n)/\St{(m,n)}$ belongs to $\partial C$ or it lies outside of $C$. For $(p,q)\in\ZZ^2$ with $p<m$ and $q<n$, denote $z_{p,q}$ the intersection point of the segment of endpoints  $(p,q)/ \St{(p,q)}$ and $(m-p,n-q)/ \St{(m-p,n-q)}$ and of the line $\Delta$ from the origin passing through $(m,n)$ (i.e the line of slope $n/m$). See Figure \ref{env} below. We get a set 
	
	$$E = \left \{ z_{p,q}, p<m, q<n \right \}$$
	
	of points of $\Delta$, that we order in the following way. We say that $z_{p,q} \preceq z_{p',q'}$ if $z_{p,q}$ is on the left of $z_{p',q'}$ on $\Delta$, when $\Delta$ is oriented positively from $(0,0)$ to $(m,n)$. By construction, 
	
	$$ z_{p,q} \preceq z_{p',q'} \iff \St{(p',q')}+\St{(m-p',n-q')} \le \St{(p,q)}+\St{(m-p,n-q)}   $$
	
	Thus finding $(p,q)$ such that $\St{(p,q)}+\St{(m-p,n-q)}$ is minimal is equivalent to finding $(p,q)$ such that $z_{p,q}$ is maximal for $\preceq$, that is the rightmost point of $E$. By definition of $C$, this is equivalent to finding $(p,q)$ such that $z_{p,q} \in \partial C$. Up to exchanging $(p,q)$ and $(m-p,n-q)$ we can assume $q/p < n/m$ (thus $ n-q/m-p > n/m$) and by convexity of $C$ we finally have:
	$$z_{p,q} \preceq z_{p',q'} \iff \frac{q}{p} \le  \frac{q'}{p'} $$
	This means that we have to find the greatest rational $q/p < n/m$ with $p<m$. We recognize the characterization of a Farey parent of $n/m$, hence $(p,q)=(b,a)$ and $(m-p,n-q)=(d,c)$. We have shown:
	$$ \min\limits_{(p,q)} \St{(p,q)}+\St{(m-p,n-q)} = \St{(b,a)} + \St{(d,c)} $$
	
	\begin{center}\label{env}
	\includegraphics[scale=0.6]{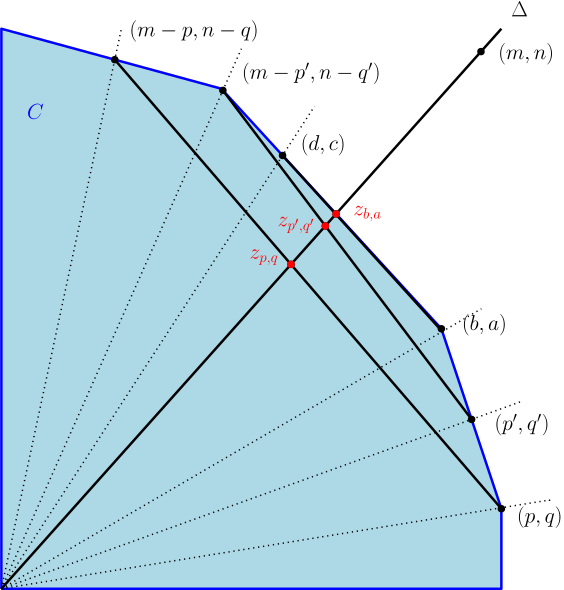}
	\captionof{figure}{The construction of the proof}
\end{center}

\end{proof}

\section{Looking for short simple closed curves}

Once again, let $h=(m,n)$ be a primitive integral homology class with $m\rho > 1$, i.e the class $h$ is not visible. We now want to know what is the shortest length of a simple curve representing $h$, in order to compare it to the length of the shortest non-simple closed curve that we found in the previous section.

We use the following result from classical algebra, which is a corollary of an old theorem by Nielsen, proven by Osborne and Zieschang as Corollary 3.2 in \cite{OZ}.

\begin{thm} Let $\psi: F_2 \longrightarrow \ZZ^2$ be the canonical abelianizing homomorphism, and let $p_1,p_2$ be two primitive elements of $F_2$. If $\psi(p_1)=\psi(p_2)$ in $\ZZ^2$ then $p_1$ and $p_2$ are conjugate in $F_2$.
\end{thm}

Recall that an element $w\in F_2$ is primitive if there exists $w'\in F_2$ such that $\langle w,w' \rangle = F_2$. Equivalently, given a system of generators $(s,t)$ of $F_2$, an element $w$ is primitive if and only if there exists $\phi \in Aut(F_2)$ such that $\phi(s)=w$. Thus on the slit torus, where for all $ x \in X, \pi_1(X,x) \cong F_2$ and $ H_1(X,\ZZ)\cong\ZZ^2$, up to conjugacy there is a one-to-one correspondance between primitive homotopy classes and integral homology classes. Since the conjugacy homotopy classes  can be identified with free homotopy classes, this means that given an integral homology class $h$ on $X$, there exists a unique free homotopy class of simple closed curves that projects onto $h$ through the abelianizing homomorphism.

\begin{lemma}
	A free homotopy class $[\gamma]$ on $X$ is primitive if and only if there exists a representative $\gamma$ that is simple.
\end{lemma}

\begin{proof}
	Denote $(s,t)$ a basis of $F_2\cong \pi_1(X)$ and let $a$ be a simple closed curve representing $s$.
	If $\gamma$ is a simple closed curve then by the theorem of classification of (non-separating) simple closed curves on surfaces (see \cite{Farb-Marg}, paragraph 1.3.1) there exists a homeomorphism $\Phi$ of $X$ such that $\Phi(a) = \gamma$. This homeomorphism induces an automorphism $\varphi$ of $\pi_1(X)$ such that $\varphi(s) = [\gamma]$, hence $[\gamma]$ is primitive. Conversely, if $[\gamma]$ is primitive there exists $\varphi\in Aut(\pi_1(X))$ such that $\varphi(s) = [\gamma]$. Since $X$ is a $K(F_2,1)$, by Proposition 1B.9 of \cite{Hatcher} $\varphi$ is induced by a homeomorphism $\Phi$ of $X$. Then, since $a$ is simple and $\Phi$ is a homeomorphism, $\Phi(a)$ is a simple representative of $[\gamma]$.
\end{proof}

Thus, when looking for simple curves representing a given integral primitive homology class $h=(m,n)$, we only need to look inside of a single free homotopy class. Osborne and Zieschang also provide a geometric construction for this class. Let $(s,t)$ be a basis of $F_2$. On a square grid, draw the open segment from $(0,0)$ to $(m,n)$. Following this segment from the origin to $(m,n)$, write $s$ (resp. $t$) whenever it crosses a vertical (resp. horizontal) line of the grid. We obtain a word $V_{m,n}'$ in the two letters $s$ and $t$.  Note that by taking an \emph{open} segment we do not count the intersections with the lines of the grid at the points $(0,0)$ and $(m,n)$ when writing this word.

\begin{thm}[Proposition 2.3 in \cite{OZ}]
	The word $V_{m,n} := stV_{m,n}'$ is primitive in $F_2$, and it is a representative of the unique conjugacy class that projects onto $(m,n)$ in $\ZZ^2$.
\end{thm}

Osborne and Zieschang also prove a remarkable property of this word $V_{m,n}$ that we rephrase here using the Farey sequence:

\begin{thm}[Lemma 2.2 in  \cite{OZ}]\label{mot}
	Let $a/b$ and $c/d$ be the Farey parents of $n/m$, with $a/b<c/d$. Then $$ V_{m,n} = V_{d,c}V_{b,a} = V_{b,a}tsV_{d,c}'$$
\end{thm}

There is a lot of information to be extracted from this theorem. The first equality $V_{m,n} = V_{d,c}V_{b,a}$ tells us that the lift of a simple curve representing $(m,n)$ on $X$ must cross the vertical line $x=d$ between the points $(d,c-1+\rho)$ and $(d,c)$, as illustrated by the following picture:

\begin{center}
	\includegraphics[scale=0.5]{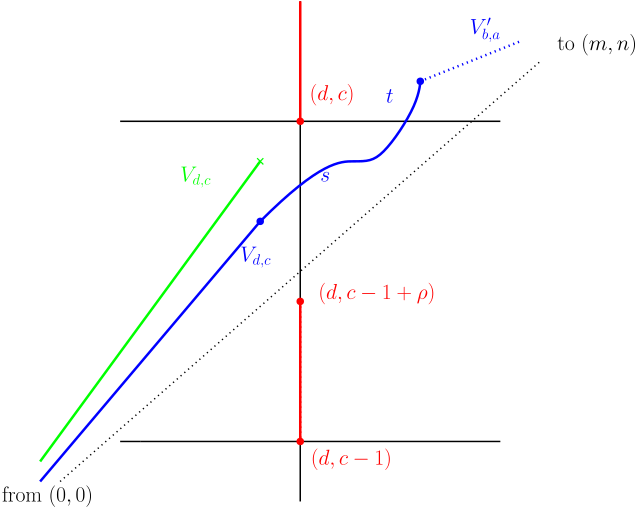}
 \captionof{figure}{First constraint for a curve in $V_{m,n}$}
	\label{contrainte 1}
\end{center}

Similarly, the second equality $V_{m,n} = V_{b,a}tsV_{d,c}'$ tells us that the lift of a simple curve representing $(m,n)$ on $X$ must cross the vertical line $x=b$ between the points $(b,a+\rho)$ and $(b,a+1)$:

\begin{center}
	\includegraphics[scale=0.45]{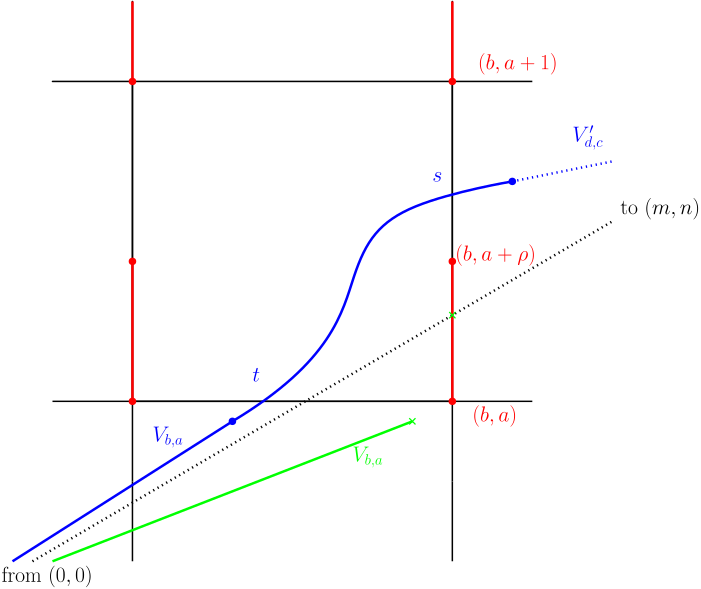}
	\captionof{figure}{Second constraint for a curve in $V_{m,n}$}
	\label{fig:contrainte 2}

\end{center}

Let us start by considering this second constraint on simple curves, illustrated by Figure \ref{fig:contrainte 2}; as we will see, it already gives a lot of information. Satisfying this constraint is a necessary condition for a curve to be simple. Among all the curves (simple or not) satisfying this constraint, the obvious candidate to be the shortest is the union $\gamma$ of the segments $[(0,0), (b,a+\rho)]$ and $[(b,a+\rho),(m,n)]$. Thus, provided that $\gamma$ projects onto a closed curve on $X$, that is to say provided that it does not encounter the slit, it will be the shortest curve satisfying the constraint illustrated by Figure \ref{fig:contrainte 2}. Its length will then provide a lower bound on the length of any curve satisfying the constraint of Figure   \ref{fig:contrainte 2}, and in particular on the length of simple curves representing $(m,n)$.  Moreover, if $\gamma$ projects onto a simple curve this lower bound would be a minimum and we would have found the shortest simple curve representing $h$. Hence we have to check if this path $\gamma$ satisfies the required properties, that is not to intersect the slit and to project onto a simple curve. There are three distinct cases to consider, corresponding to the various possibilities for the Farey parents of $n/m$ to be associated to visible homology classes or not.

\subsection{Both Farey parents are visible }

Let $a/b$ and $c/d$ be the Farey parents of $n/m$, with $a/b<c/d$, and assume that $ b\rho \le 1$ and $d\rho\le 1$, where $\rho$ is the length of the slit. This means that the classes $(b,a)$ and $(d,c)$ are visible: in that sense we say that the Farey parents of $n/m$ are visible.

\begin{lemma}
Let $\gamma$ be the path in $\tilde{X}$ from $(0,0)$ to $(m,n)$, obtained by the concatenation of the segments  $[(0,0), (b,a+\rho)]$ and $[(b,a+\rho),(m,n)]$.
	If both Farey parents of $n/m$ are visible then the path $\gamma$ does not intersect the slit.
\end{lemma}

\begin{proof}
	We check that the two segments composing $\gamma$ do not intersect the slit.
	
	\begin{itemize}\renewcommand{\labelitemi}{$\bullet$}
		
		\item The segment from $(0,0)$ to $(b,a+\rho)$ has slope $(a+\rho)/b$, it crosses the vertical of the slit at the points $(k,k \mathlarger{\frac{a+\rho}{b}})$ with $k=1,...,b-1$. 
		An intersection with the slit occurs at such a point if and only if 
		$$\left \{ k\frac{a+\rho}{b} \right \} < \rho$$
		where $\{. \}$ denotes the fractional part. We have:
		$$\left \{ k\frac{a+\rho}{b} \right \}  = \left \{\frac{r}{b} +\frac{k\rho}{b} \right \} $$
		where $0\le r< b$ is the remainder in the Euclidean division of $ka$ by $b$. Let $L = \lfloor 1/\rho \rfloor$ be the integral part of $1/\rho$. We have:
		
		$$ \frac{r}{b}+\frac{k \rho}{b}  \le \frac{r+\mathlarger{\frac{k}{L}}}{b}
		\\ \le \frac{r+\mathlarger{\frac{b-1}{L}}}{b} 
		\\ \le \frac{r+1}{b} 
		\\ \le 1
		$$
		because $\rho \le 1/L$, $k\le b-1$, $b\le L$ and $r\le b-1$. Hence $ \left \{ \mathlarger{\frac{r}{b} + \frac{k\rho}{b}} \right \} = \mathlarger{\frac{r}{b} + \frac{k\rho}{b}}$. Moreover,
		
		$$\frac{r}{b} + \frac{k\rho}{b} \ge \frac{1+k\rho}{b} > \frac{1}{b} \ge \frac{1}{L} \ge \rho$$
		
		so for any $k=1,...,b-1$,
		
		$$\left  \{ k\frac{a+\rho}{b}  \right \} \ge \rho$$
		
		and the segment $[(0,0),(b,a+\rho)]$ does not intersect the slit.
		
		\item For the case of the segment from $(b,a+\rho)$ to $(m,n)$, the computations become easier when we translate the problem by $(-b,-a-\rho)$. Since $m=b+d$ and $n=a+c$, the problem is equivalent to determining whether or not the segment $[(0,0),(d,c-\rho)]$ intersects the slit, with the copies of the slit being this time the sets $\left \{ (p,q-s), s\in [0,\rho] \right \}$. The segment $[(0,0),(d,c-\rho)]$ has slope $(c-\rho)/d$ and crosses the vertical of the slit at the points $(k,k\mathlarger{\frac{c-\rho}{d}})$, $k=1,...,d-1$, with intersection with the slit if and only if
		$$ \left  \{ k \frac{c-\rho}{d} \right \} > 1-\rho $$
		
		Let $r$ be the remainder in the Euclidean division of $kc$ by $d$. We have $\left \{ k \mathlarger{\frac{c-\rho}{d}}\right \}=\left \{ \mathlarger{\frac{r-k\rho}{d}}\right \}$. Moreover,
		
		$$\frac{r-k\rho}{d} > \frac{1-k\rho}{d} \ge \frac{1-(d-1)\rho}{d} > 0$$ 
		
		because $(d-1)\rho \le 1 - \rho < 1$, and since $1/d > \rho$ we have
		
		$$ \frac{r-k\rho}{d} < \frac{d-1-k\rho}{d} = 1-\frac{1+k\rho}{d} < 1- \frac{1}{d} < 1-\rho$$ 
		
		 Hence the segment $[(0,0),(d,c-\rho)]$  does not intersect the slit.

	\end{itemize}
\end{proof}

Thus the path $\gamma$ projects onto a closed curve on $X$ that represents $h$. It is the shortest curve satisfying the constraint illustrated by Figure \ref{fig:contrainte 2}. So in particular, if this curve is simple, it is the shortest simple closed curve representing $h$.

\begin{prop}\label{2 parent}
	Using the previous notations, the path $\gamma$ projects onto a simple closed curve. This curve is minimizing for $h=(m,n)$ and has the same length as $\gamma$, so
	$$\St{(m,n)} = l(\gamma)= \sqrt{b^2+(a+\rho)^2}+\sqrt{d^2+(c-\rho)^2}$$
	where $n/m = a/b \oplus c/d$ with $b,d \le 1/\rho$. Since minimizing curves for $h$ are simple, the unit ball of the stable norm is strictly convex in the direction $(m,n)$.
\end{prop}

\begin{proof}
	 We know (see Proposition \ref{non simple}) that the shortest non-simple curve representing $h$ has length $\St{(b,a)} + \St{(d,c)}$, and since we assumed that both $(b,a)$ and $(d,c)$ are visible by Proposition \ref{visible} we know what their stable norm is. Thus the shortest non-simple curve representing $h$ has length $$\St{(b,a)} + \St{(d,c)} = \sqrt{b^2+a^2} + \sqrt{d^2 + c^2}$$ and we only need to show that $\gamma$ projects onto a shorter curve: if so it cannot be non-simple, thus necessarily it has to be simple. Note that we could show directly that $\gamma$ projects onto a simple curve but we would still have to give the exact same following proof in order to show it is minimizing. This way, we get that $\gamma$ projects onto a simple curve for free. \newline
  
  Since $a/b$ and $c/d$ are Farey neighbours they satisfy the relation $bc-ad=1$, so
	
	$$\frac{n}{m} - \frac{a+\rho}{b} = \frac{a+c}{b+d}-\frac{a}{b} - \frac{\rho}{b}= \frac{1 }{b}\left (\frac{1}{b+d} - \rho \right ) < 0$$
	
	because $b+d=m > 1/\rho$ and since $d\rho \le 1$ we have $\rho\le 1/d$ and:
	$$\frac{a+\rho}{b} = \frac{a}{b} + \frac{\rho}{b} \le  \frac{a}{b} + \frac{1}{bd} = \frac{c}{d}$$
	
	hence we have $\mathlarger{\frac{c-\rho}{d} < \frac{a}{b} < \frac{n}{m} < \frac{a+\rho}{b} < \frac{c}{d}}$. Thus the triangle formed by the points $(0,0)$, $(d,c)$ and $(m,n)$ contains the triangle formed by the points $(0,0), (b,a+\rho)$ and $(m,n)$ (see Figure \ref{triangle inclus}). We deduce that the perimeter of the first triangle is greater than the perimeter of the second. But their perimeters are respectively $$\St{(b,a)} + \St{(d,c)} + \sqrt{m^2+n^2}$$ and $$ \sqrt{b^2+(a+\rho)^2}+\sqrt{d^2+(c-\rho)^2} + \sqrt{m^2+n^2} = l(\gamma) + \sqrt{m^2+n^2} $$ so we finally have $l(\gamma) < \St{(b,a)} + \St{(d,c)}$, hence $\gamma$ projects onto a simple curve on $X$.
	\begin{center}
		\includegraphics[scale=0.8]{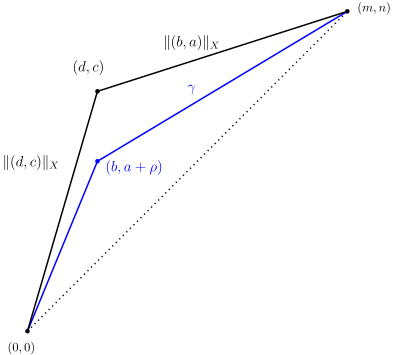}
		\captionof{figure}{The path $\gamma$ is shorter}
		\label{triangle inclus}
	\end{center}
But we know (see discussion below Figure \ref{fig:contrainte 2}) the length of $\gamma$ is a lower bound on the lengths of simple curves representing $h$. Since $\gamma$ projects onto a simple curve, this means that the shortest simple curve representing $h$ has length $l(\gamma)=\sqrt{b^2+(a+\rho)^2}+\sqrt{d^2+(c-\rho)^2} $. Moreover, since it is shorter than the length of the shortest non-simple curve representing $h$, the path $\gamma$ projects onto a minimizing curve for $h$, and $l(\gamma) = \St{h}$.

\end{proof}

Remark that we did not even have to worry about whether or not $\gamma$ satisfies the constraint illustrated by Figure \ref{contrainte 1}, because we showed the simplicity of its projection onto $X$ more directly. However if one feels the need to double check, a direct computation shows that $\gamma$ does indeed satisfy the condition to cross the vertical $x=d$ between $(d,c-1+\rho)$ and $(d,c)$, as it is expected for a simple curve.

\subsection{One of the Farey parents is visible}

This time let us assume that exactly one of the Farey parents of $n/m$ is visible, i.e if $a/b < c/d$ denote the Farey parents of $n/m$ we have $(b\rho\le 1$ and $d \rho> 1)$ or $(b\rho >1$ and $d \rho \le 1)$. Although the computations are slightly different, we get the same results as in the case where both Farey parents are visible.

\begin{lemma}
If one of the Farey parents of $n/m$ is visible then the path $\gamma$, composed of the segments $[(0,0),(b,a+\rho)]$ and $[(b,a+\rho),(m,n)]$, does not intersect the slit. Hence it projects onto a closed curve on $X$.
\end{lemma}

\begin{proof}
	In order to be able to do explicit computations, let us assume that $d\rho \le 1$ and $b\rho  >1$; the proof is identical if we assume $d\rho  > 1$ and $b\rho\le 1$. The same proof as in the previous part still holds to show that the segment $[(b,a+\rho),(m,n)]$ does not intersect the slit, but if fails for the other segment since now $1/b < \rho$.
	\\ \\
	The segment $[(0,0), (b,a+\rho)]$ crosses the vertical of the slit at points $\left(k,k\mathlarger{\frac{a+\rho}{b}}\right)$, $k=1,...,b-1$, with intersection with the slit if and only if $\left \{ k\mathlarger{\frac{a+\rho}{b}} \right \} < \rho$. Since $a/b$ and $c/d$ are a Farey pair, we have 
	$$ \frac{a}{b} = \frac{c}{d} - \frac{1}{bd} $$
	so
	
	$$\left \{ k \frac{a+\rho}{b} \right \} =  \left \{ k \left(\frac{c}{d}-\frac{1}{bd} + \frac{\rho}{b} \right)\right \}=\left \{ \frac{r}{d}+k\frac{\rho-1/d}{b} \right \}$$
	
	where $r$ is the remainder in the Euclidean division of $kc$ by $d$. Remark that $d\rho\le 1 $ so $\rho - 1/d <0$. On one hand
	
	$$\frac{r}{d}+k\frac{\rho-1/d}{b} \le \frac{r}{d} \le \frac{d-1}{d} < 1$$ 
	
	and on the other hand
	
	$$\frac{r}{d}+k\frac{\rho-1/d}{b} \ge \frac{1}{d} + \frac{k(\rho-1/d)}{b} \ge \frac{1}{d} + \frac{(b-1)(\rho-1/d)}{b}$$
	
	so 
	
	$$\frac{r}{d}+k\frac{\rho-1/d}{b} \ge \rho + \frac{1/d - \rho}{b}> \rho$$
	
	Thus $\mathlarger{\left \{ \frac{r}{d}+k\frac{\rho-1/d}{b} \right \} = \frac{r}{d}+k\frac{\rho-1/d}{b}} > \rho$ and the segment $[(b,a+\rho),(m,n)]$ does not intersect the slit.
	
\end{proof}

\begin{prop}\label{1 parent}
	If one of the Farey parents $a/b<c/d$ of $n/m$ is visible, the path $\gamma$ defined above projects onto a simple closed minimizing curve for $h$, and we have
	$$\St{(m,n)} = \sqrt{b^2+(a+\rho)^2}+\sqrt{d^2+(c-\rho)^2}$$
	Moreover, the unit ball of the stable norm is strictly convex in the direction $(m,n)$.
\end{prop}

\begin{proof}
	 Assume for instance $b\rho \le 1$ and $d\rho > 1$, as the proof is similar in the other case. Then $c/d$ and $a/b$ are Farey neighbours in some Farey sequence, and since $d > b$ we get that $a/b$ is one of the Farey parents of $c/d$. Let $\alpha/\beta$ be the other Farey parent of $c/d$, i.e $c/d = a/b \oplus \alpha/\beta$.\\ \\
	 If $\alpha / \beta$ is visible, i.e if $\beta \rho \le 1$, then by proposition \ref{2 parent} we can find a minimizing curve for $(d,c)$ and the same proof as in \ref{2 parent} holds when replacing the segment $[(0,0),(d,c)]$ by this curve, as illustrated in Figure \ref{inclus}.
	 
	 \begin{center}
	 	\includegraphics[scale=0.8]{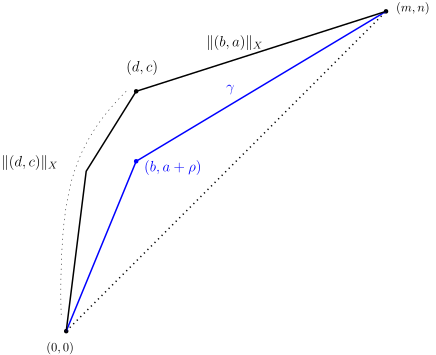}
   \captionof{figure}{The path $\gamma$ is shorter}
	 	
	 	\label{inclus}
	 \end{center}
 
 If $\alpha / \beta$ is not visible, i.e if $\beta \rho > 1$, again $\alpha/\beta$ must be a Farey child of $a/b$. By considering the second Farey parent of $\alpha/\beta$ and repeating the process, we obtain a sequence of rationals $\alpha_0/\beta_0 , ..., \alpha_{k-1}/\beta_{k-1}, \alpha_k/\beta_k = \alpha/\beta$ such that $\alpha_0/\beta_0$ is neighbour to $a/b$ in some Farey sequence and visible, and $a_{i+1}/\beta_{i+1} = \alpha_i /\beta_i \oplus a/b$. According to the previous point, we know what is a minimizing curve for $(\beta_1,\alpha_1)$, so by repeating this argument we find a minimizing curve for $(\beta_2,\alpha_2)$, and by induction we finally obtain the same situation as described by Figure \ref{inclus} if $\alpha/\beta$, and the same argument as before completes the proof.
\end{proof}

\subsection{None of the Farey parents are visible}

Finally, what happens when none of the Farey parents $a/b$ and $c/d$ of $n/m$ are visible, that is to say when $b,d > 1/\rho$ ? This time, the path which is the union of the segments $[(0,0),(b,a+\rho)]$ and $[(b,a+\rho),(m,n)]$ does not project onto a simple closed curve on the slit torus. Indeed it intersects the slit, for instance at $x=d$ as one can check with a simple computation that we will not detail. Instead we will prove a more general fact:

\begin{prop}\label{no simple}
	If none of the Farey parents of $n/m$ are visible, there is no simple geodesic in the free homotopy class $V_{m,n}$.
\end{prop}

First we need the following lemma: 

\begin{lemma}
	Let $\gamma$ be the lift to $\tilde{X}$ of a geodesic in the free homotopy class $V_{m,n}$ starting at the origin. Then $\gamma$ changes direction at least twice.
\end{lemma}

\begin{proof}
	Let $a/b < c/d$ be the Farey parents of $n/m$, by assumption $b,d > 1/\rho$.
	\begin{itemize}\renewcommand{\labelitemi}{$\bullet$}
		\item If $d<b$, since $d> 1/\rho$ and $bc-ad=1$, straightforward computations yield
  $$\frac{c-1+\rho}{d}<\frac{a}{b}<\frac{n}{m} <\frac{c}{d} < \frac{a+\rho}{b}$$
  so by putting together the constraints illustrated by Figure \ref{contrainte 1} and Figure \ref{fig:contrainte 2}, we obtain the following picture, on which it is clear that the lift of a curve in $V_{m,n}$ must change direction at least twice.

  \begin{center}
      \includegraphics[scale=0.7]{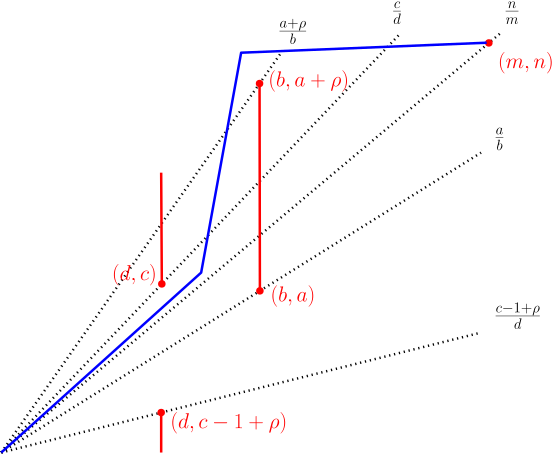}
      \captionof{figure}{The blue curve in $V_{m,n}$ must change direction twice}
      \label{change direction 1}
  \end{center}

  \item If $b<d$, similarly we get the following picture:

  \begin{center}
 \includegraphics[scale=0.7]{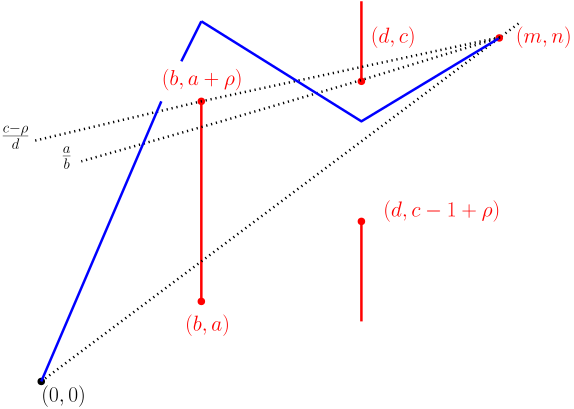}
 \captionof{figure}{Again, the blue curve in $V_{m,n}$ must change direction twice}
  \end{center}

  and since $b>1/\rho$ we have $\mathlarger{\frac{c-\rho}{d}<\frac{a}{b}}$, so the lift of a curve in $V_{m,n}$ has to change direction in order to pass under $(d,c)$.
	\end{itemize}

\end{proof}

We can now prove Proposition \ref{no simple}.

\begin{proof}
    Since the metric is flat, a geodesic on $X$ must be a broken line with angles at the slit's endpoints. Indeed, if a geodesic changes direction outside of an endpoint of the slit, then it has a corner and this is not locally length-minimizing, as it should be for a geodesic. Since we know that every lift of a geodesic in $V_{m,n}$ must change direction at least twice between $(0,0)$ and $(m,n)$, that means that any geodesic in $V_{m,n}$ passes at least three times through an endpoint of the slit: the two times we have just found, and the third time corresponding to start and finish, i.e the projection of the points $(0,0)$ and $(m,n)$. The slit has only two endpoints, hence any geodesic in $V_{m,n}$ meets an single endpoint of the slit at least twice, so it is not simple.
\end{proof}

Obviously, a minimizing curve for a homology class $h$ has to be a geodesic. Thus, if there are no simple geodesic representing $h$ this means that $h$ is minimized by a non-simple curve, and we know its length according to Proposition \ref{non simple}.
We have shown:

\begin{prop}\label{0 parent}
    Let $h=(m,n)$ be a primitive integral homology class, and let $a/b,c/d$ be the Farey parents of $n/m$. Assume that none of the Farey parents are visible, that is to say $b,d> 1/\rho$ where $\rho$ is the length of the slit. Then $\St{(m,n)}=\St{(b,a)} + \St{(d,c)}$ and the unit ball $\Bb$ of the stable norm has a flat in the direction $h$.
\end{prop}

\section{Stable norm of slit tori}
\subsection{Statement of the theorem}

We can finally state our main result. Putting together Proposition \ref{visible}, \ref{non simple}, \ref{2 parent}, \ref{1 parent} and \ref{0 parent} we have completely computed the stable norm of the slit torus $X$.

\begin{thm}\label{main thm}
Let $h=(m,n)\in H_1(X,\ZZ)$ be a primitive integral homology class, and let $a/b$ and $c/d$ be the Farey parents of $n/m$, with $a/b<c/d$. The unit ball $\Bb$ of the stable norm is strictly convex in the direction $h$ if and only if either

\begin{itemize}\renewcommand{\labelitemi}{$\bullet$}
\item $|m\rho| \le 1$, and in that case $\St{(m,n)}=\sqrt{m^2+n^2} = \left\| \begin{pmatrix} m \\ n \end{pmatrix} \right \|_2$  \\ \\
or
\item $|m\rho| >1$ and $(|b\rho|\le 1$ or $|d\rho|\le 1)$, and in that case $$\St{(m,n)}=\sqrt{b^2+(a+\rho)^2}+\sqrt{d^2+(c-\rho)^2}= \left \|\begin{pmatrix} b \\ a \end{pmatrix} + \begin{pmatrix} 0 \\ \rho \end{pmatrix} \right \|_2 + \left \| \begin{pmatrix} d \\ c \end{pmatrix} - \begin{pmatrix} 0 \\\rho\end{pmatrix} \right \|_2$$ 
\end{itemize}
Moreover, $\Bb$ has a vertex in those directions.
The unit ball $\Bb$ has a flat in every other direction of $H_1(X,\RR)$. In particular, if $|m|,|b|,|d|> 1/\rho$ we have $\St{(m,n)}=\St{(b,a)}+\St{(d,c)}$.
\end{thm}

\begin{proof}
    The only thing left to prove is the claim about strictly convex directions being vertices. We say that an extreme point $p$ of a convex set $C$ in the plane is a vertex if it admits infinitely many supporting hyperplanes, i.e a hyperplane that contains $p$ and such that $C$ is entirely contained in one of the two half-spaces bounded by the hyperplane. \newline 

   Let us start with the easier case. Let $(m,n)$ be a strictly convex direction of $\Bb$, and assume that $|m\rho| > 1$: this correspond to the second item in the theorem.  The direction $(m,n)$ is isolated from other strictly convex directions of the stable norm. Indeed, its slope is isolated from the slopes of other strictly convex directions, because successive Farey children of a rational are isolated from one another. Thus the point $(m,n) / \St{(m,n)}$ is a common endpoint of two flats of $\Bb$, i.e it is an endpoint of two segments contained in $\partial \Bb$. These two segments form an angle at their common endpoint $(m,n) / \St{(m,n)}$: indeed, if $p_1$ and $p_2$ denote the other two endpoints of the segments, easy computations show that $p_1, (m,n) / \St{(m,n)}$ and $p_2$ are not aligned. Thus there are infinitely many supporting lines at $(m,n) / \St{(m,n)}$, so it is a vertex of $\Bb$. \newline 

   Now, let $(m,n)$ be a visible direction. According to the previous point, we know there are infinitely many vertices of $\Bb$ that accumulate to $p:= (m,n) / \St{(m,n)}$: it might happen that there is only one supporting line at $p$, namely the line tangent to the circle at $p$. Those vertices accumulating on $p$ correspond to the Farey children of $n/m$, so they are of the form 
   $$u_k:= \frac{(\beta +km, \alpha + kn)}{\St{(\beta +km, \alpha + kn)}}$$
   on the left side of $p$ and 
   $$v_k:= \frac{(\delta +km, \gamma + kn)}{\St{(\delta +km, \gamma + kn)}}$$ on the right side of $p$, 
   where $k\in \NN^*$ and $(\beta,\alpha)$ and $(\delta,\gamma)$ are two visible classes neighbour to $(m,n)$. Let $\Delta_k = p - u_k$. Since $(m,n)$ is visible, we have $\St{(m,n)} = \sqrt{m^2 + n^2}$, and since $u_k$ corresponds to a direction that is a Farey child of $n/m$, we have
   $$\St{(\beta +km, \alpha + kn)} = \sqrt{ m^2 + (n \pm \rho)^2} + \sqrt{ (\beta + (k-1)m)^2 + (\alpha + (k-1)n \mp \rho)^2}$$
Direct computations then yields 
   $$\Delta_k = (\frac{c_1}{k} + O(1/k^2) , \frac{c_2}{k} + O(1/k^2))$$
   where $c_1$ and $c_2$ are non-zero constants depending only on $m,n,\alpha, \beta $ and $\rho$. Thus the slopes of the vectors $\Delta_k$ goes to the non-zero constant $c_2/c_1$ as $k$ goes to infinity: in particular those slopes are bounded away from zero. Hence any line passing through $p$ with a slope in between $0$ and the slopes of the $\Delta_k$ does not intersect $\Bb$ outside of $p$: it is a supporting line of $\Bb$ at $p$, thus $p$ is a vertex of $\Bb$.

   \begin{center}
       \includegraphics[scale=0.8]{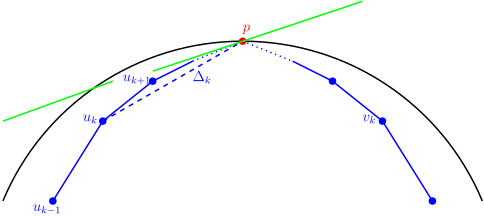}
       \captionof{figure}{The green line is a supporting line of $\Bb$ at $p$}
   \end{center}
\end{proof}

It is difficult to produce a good picture of the unit ball of the stable norm of the slit torus. Indeed, it is so close to the unit circle that they are barely distinguishable from one another by eyesight; moreover, the angles at the vertices are nearly flat so they appear almost invisible. In order to grasp the structure of $\Bb$, we flatten the arc of the unit circle between angles $0$ and $\pi /4$  onto the segment $[0,1]$, and to a rational $n/m$ we associate the distance between the unit circle $\SS^1$ and the unit ball of the stable norm $\Bb$ in the direction $(m,n)$. The points of non-differentiability of the resulting curve correspond to the strictly convex directions of $\Bb$. On the left figure, with $1/3 < \rho < 1/2$, we can clearly see the three visible directions, namely $(1,0), (2,1)$ and $(1,1)$. The right figure, with $1/6< \rho < 1/5$, highlights the accumulation of vertices on the visible direction $(1,0)$.
\begin{center}
    \includegraphics[scale=0.7]{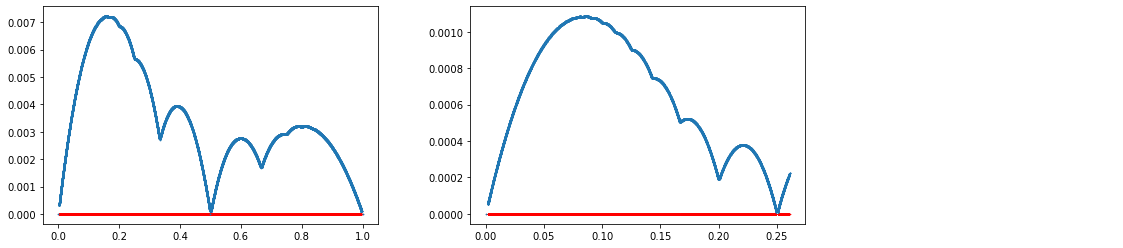}
    \captionof{figure}{Unit ball of the stable norm (blue) near an arc of the unit circle (red).}
\end{center}

\subsection{Stable norm of general slit tori}

We know how to compute to stable norm of a \emph{square} flat torus with a \emph{vertical} slit, but what about the stable norm of general slit tori? That is, flat tori whose metric comes from a general parallelogram, and with a slit of any direction.

\subsubsection{$SL(2,\RR)$-action on slit tori}

Invertible matrices act linearly on $\RR^2$ and transform the polygonal model of the flat square torus $X$ with a vertical slit (Figure \ref{modele}) into a parallelogram with a slit parallel to a side of the parallelogram. Glueing the opposites sides of the parallelogram together, we obtain a new flat torus that we denote $MX$. In general, the flat metric of $MX$ does not come from a square. More formally, an invertible matrix $M\in GL(2,\RR)$ induces a diffeomorphism $M : X \longrightarrow MX$ between the slit tori $X$ and $MX$.

\begin{center}
    \includegraphics[scale=0.8]{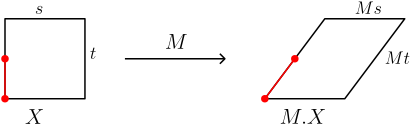}
    \captionof{figure}{Diffeomorphism between slit tori induced by $M$}
\end{center}

We want to compute the stable norm of $MX$. Since the stable norm transforms homogeneously with respect to the scaling of the surface, we can assume that $MX$ has area $1$, i.e we can only consider determinant $1$ matrices of $SL(2,\RR)$. The sides of the parallelogram modelling $MX$ give a natural basis for $H_1(MX,\ZZ)$ illustrated by the above figure. As $M$ maps the sides of the square to the sides of the parallelogram, with this choice of basis the application between homology groups $M_* : H_1(X,\ZZ) \longrightarrow H_1(MX,\ZZ)$ induced by the diffeomorphism $M$ is the identity. We get a slightly more general version of Theorem \ref{main thm}:

\begin{cor}\label{action} Let $X$ be a square torus with a vertical slit of length $\rho$. Let $M\in SL(2,\RR)$ be a determinant $1$ matrix. The unit ball $\Bb$ of the stable norm of the slit torus $MX$ has a vertex in the direction $h=(m,n)\in H_1(MX,\ZZ)$ if and only if either

\begin{itemize}\renewcommand{\labelitemi}{$\bullet$}
\item $|m\rho| \le 1$, and in that case $\|(m,n)\|_{MX}=\left \| M.\begin{pmatrix} m \\ n \end{pmatrix} \right \|_2$ \\ \\
or
\item $|m\rho| >1$ and $|b\rho|\le 1$ or $|d\rho| \le 1$, and in that case $$\|(m,n)\|_{MX}= \left \| M. \left [\begin{pmatrix} b \\ a \end{pmatrix} + \begin{pmatrix} 0 \\ \rho \end{pmatrix} \right ] \right \|_2 + \left \| M.\left [ \begin{pmatrix} d \\ c \end{pmatrix} - \begin{pmatrix} 0 \\\rho\end{pmatrix} \right ] \right \|_2$$ 
\end{itemize}
The unit ball $\Bb$ has a flat in every other direction of $H_1(MX,\RR)$. In particular, if $|m|,|b|,|d|> 1/\rho$ we have $\|(m,n)\|_{MX}=\|(b,a)\|_{MX}+\|(d,c)\|_{MX}$.

\end{cor}

\begin{proof}

The proof is exactly the same as in the case of a square slit torus with a vertical slit, up to shearing all the pictures under the linear action of $M$. In fact, $M$ maps minimizing curves on $X$ to minimizing curves on $MX$. More precisely,
\begin{itemize}\renewcommand{\labelitemi}{$\bullet$}
\item As the linear action of $M$ obviously preserves straight line paths, $(m,n)\in H_1(MX,\ZZ)$ is visible on $MX$ if and only if $(m,n)\in H_1(X,\ZZ)$ is visible, that is if and only if $|m\rho|\le 1$. Hence, in that case we have $\| (m,n)\|_{MX} = \left \|\begin{pmatrix} m \\ n \end{pmatrix} \right \| _ 2$. 
\end{itemize}

 If $(m,n)$ is not visible on $MX$, with $a/b<c/d$ the Farey parents of $n/m$:

\begin{itemize}\renewcommand{\labelitemi}{$\bullet$}
\item As it was purely combinatorial, the argument used in Proposition \ref{non simple} still holds in this new setting. Hence the shortest non-simple curve representing $(m,n)$ on $MX$ has length $\|(b,a)\|_{MX}+\|(d,c)\|_{MX}$.

\item Since $M$ is a diffeomorphism, it maps bijectively simple curves on $X$ to simple curves on $MX$. Thus, if $(b,a)$ or $(d,c)$ is visible we get the shortest simple curve $\gamma$ representing $(m,n)$ in the exact same way as in Proposition \ref{2 parent} and Proposition \ref{1 parent}. It is the image by $M$ of the minimizing curve of $(m,n)$ on $X$. The action of $M$ shears Figure \ref{triangle inclus} and Figure \ref{inclus} but does not change how lengths compare to one another, so $\gamma$ is minimizing for $(m,n)$. Hence $$ \|(m,n)\|_{MX}= l(\gamma) = \left \| M. \left [\begin{pmatrix} b \\ a \end{pmatrix} + \begin{pmatrix} 0 \\ \rho \end{pmatrix} \right ] \right \|_2 + \left \| M.\left [ \begin{pmatrix} d \\ c \end{pmatrix} - \begin{pmatrix} 0 \\\rho\end{pmatrix} \right ] \right \|_2 $$

\item Again, since $M$ maps simple curves to simple curves, if none of $(b,a)$ and $(d,c)$ are visible, since there are no simple geodesics representing $(m,n)$ on $X$, there are no simple geodesics representing $(m,n)$ on $MX$. So the unit ball of the stable norm of $MX$ has a flat in the direction $(m,n)$.

\end{itemize}

\end{proof}

Remark that although the strictly convex directions are the same for $X$ and $MX$ the lengths do not transform linearly, so the unit ball of the stable norm of $MX$ is \emph{not} the image of the unit ball of $X$ under the linear action of $M$ on $H_1(X,\RR) \cong \RR^2 $. 

\subsubsection{Rational slit tori}

In this section we discuss technicalities that will greatly simplify the arguments in the next section.
We now know how to compute the stable norm of a slit torus modelled by a parallelogram, assuming the slit is parallel to one side of the parallelogram. What happens when this is not the case? This time, we will only consider the case of square tori: once this is understood, by a similar argument as in Corollary \ref{action} we will be able to get the stable norm of general parallelogram tori by applying $SL(2,\RR)$ matrices to the square tori. \newline  

So let $X'$ be a square slit torus, and assume that the slit is not parallel to any side of the square. We represent the slit with the vector $(\beta, \alpha)$: more precisely, if the lower endpoint of a lift of the slit is placed at the origin in the abelian covering $\tilde{X}'$, the upper endpoint has coordinates $(\beta,\alpha)$. Moreover, assume  $\alpha / \beta\in \QQ$, i.e the slit has rational direction. We say that $X'$ is a rational slit torus; the case of irrational tori is the object of the next section. \newline

The linear action of $SL(2,\ZZ)$ on $\RR^2$ preserves the lattice $\ZZ^2$. Thus by applying a $SL(2,\ZZ)$ matrix to $X'$ we get a parallelogram that can be rearranged into a square by cut and paste operations: these operations on the polygonal model do not change the metric of the underlying surface. In other words, the image of $X'$ under a diffeomorphism induced by an $SL(2,\ZZ)$ matrix is again a square slit torus, whose slit possibly has different slope than the one of $X'$. Since we assumed $X'$ has a slit of rational slope, by applying the adequate matrix of $SL(2,\ZZ)$ to $X'$ we can obtain a square torus with a vertical slit. Or the other way around, there exists $M\in SL(2,\ZZ)$ such that $X' = MX$ where $X$ is a square torus with a vertical slit.

\begin{center}
    \includegraphics{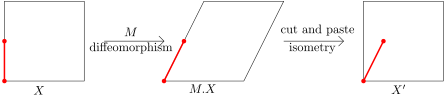}
    \captionof{figure}{Operations on $X$ leading to $X'$ }
\end{center}

Thus, rational tori are a particular case of Corollary \ref{action}. So why should we bother? The main difference between $X'$ and $MX$ is the choice of a basis for the first homology group. In both cases, the natural choice is given by the sides of their polygonal model. However, on $MX$ the slit is parallel to a side of the parallelogram, which is no longer the case of $X'$. Far from anecdotal, this fact means that on $X'$ our natural choice of a basis of the homology does not depend on the direction of the slit. This fact allows us to have a common homology basis of reference when considering several rational slit tori, as we will do in the next section. Although it may look like we did not really do anything, we can now reformulate the qualitative statements of Corollary \ref{action} on $X'$ in a intrinsic way; that is, without any reference to $X$.

\begin{cor}\label{rational}
    Let $X'$ be a rational square torus with a slit represented by the vector $(\beta, \alpha)$ with $\alpha / \beta \in \QQ$. Take the sides of the square representing $X'$ as a basis of $H_1(X',\ZZ)$. The unit ball of the stable norm of $X'$ has a vertex in the direction $(m,n)\in H_1(X',\ZZ)$ if and only if either 

\begin{itemize}\renewcommand{\labelitemi}{$\bullet$}
\item $(m,n)$ is visible, which is equivalent to $|m\alpha - n\beta| \le 1$. \\ \\
or 
\item $|m\alpha - n\beta|> 1$ and there exists $(m',n')\in H_1(X',\ZZ)$ visible such that $n'/m'$ and $n/m$ are neighbours in some Farey sequence. 

\end{itemize}
   The unit ball of the stable norm of $X'$ has a flat in all other directions of $H_1(X',\ZZ)$.
\end{cor}

\begin{proof}
We know there exists a square torus $X$ with a vertical slit of length $\rho$ and a matrix $M\in SL(2,\ZZ)$ such that $X'$ is isometric to $MX$.
Note that the cut and paste operation from $MX$ to $X'$ is an isometry $\varphi : MX \longrightarrow X'$ whose induced map in homology $\varphi_* : H_1(MX,\ZZ) \longrightarrow H_1(X',\ZZ)$ acts as $M^{-1}$. Thus we only need to apply $M^{-1}$ to the homology classes mentioned in the statements of Corollary \ref{action} in order to get the result. 
\newline

Hence, a class $(m,n)$ is visible on $X'$ if and only if $M^{-1} \begin{pmatrix} m \\ n \end{pmatrix}$ is visible on $MX$ (or equivalently on $X$). We know $M$ sends the slit $\begin{pmatrix} 0 \\ \rho \end{pmatrix}$ of $X$ to the slit $\begin{pmatrix}
    \beta \\ \alpha
\end{pmatrix}$ of $X'$. But $\alpha / \beta$ is rational so there exist $p,q\in \ZZ$ with $gcd(p,q)=1$ such that $\alpha / \beta = p/ q$. Thus $M$ is of the form $\begin{pmatrix}
    v & q \\ u & p 
\end{pmatrix}$ with $u,v\in \ZZ$, and $M^{-1}$ is of the form $\begin{pmatrix}
    p & -q \\ -u & v
\end{pmatrix}$.
We have 
$$M^{-1} \begin{pmatrix} m \\ n \end{pmatrix} = \begin{pmatrix}
    mp-nq \\ vn-um
\end{pmatrix}$$

so by Corollary \ref{action}, $M^{-1} \begin{pmatrix} m \\ n \end{pmatrix}$ is visible on $MX$ if and only if $|(mp-nq)\rho| \le 1$, that is to say if and only if $|m\alpha - n\beta|\le 1$ since $M\begin{pmatrix} 0 \\ \rho \end{pmatrix} = \begin{pmatrix} \beta \\ \alpha \end{pmatrix}$. \newline

If a class $(m,n)$ is not visible on $X'$, it is a strictly convex direction of the unit ball of the stable norm of $X'$ if and only if $M^{-1} \begin{pmatrix} m \\ n \end{pmatrix}= \begin{pmatrix}
    mp-nq \\ vn-um
\end{pmatrix}$ is a strictly convex direction of the unit ball of the stable norm of $MX$.
This is the case if and only if at least one of the Farey parents of $(mp-nq)/(vn-um)$ is visible. In particular this Farey parent, denoted $n''/m''$, is neighbour to $(mp-nq)/(vn-um)$ in some Farey sequence, i.e $n''(vn-um) - m''(mp-nq)=1$ or in other words
$$\det \begin{pmatrix}
    m'' & mp-nq \\ n'' & vn-um
\end{pmatrix} = 1$$
Going back to $X'$ by applying $M$, let $\begin{pmatrix}
    m' \\ n'
\end{pmatrix} = M\begin{pmatrix}
    m'' \\ n''
\end{pmatrix}$. Since $\det M =1$ we have

$$\det \begin{pmatrix} m' & m \\ n' & n \end{pmatrix} = \det \left [ M. \begin{pmatrix}
    m'' & mp-nq \\ n'' & vn-um
\end{pmatrix} \right ] = 1$$
so $n/m$ and $n'/m'$ are Farey neighbours, hence the result. Note that even though $n'/m'$ is neighbour to $n/m$, it is not necessarily a Farey parent of $n/m$. \newline 

Finally, if $(m,n)$ does not have any visible neighbour on $X'$, then it is also the case for $M^{-1} \begin{pmatrix} m \\ n \end{pmatrix}$ on $MX$. In particular, the Farey parents of $(mp-nq)/(vn-um)$ are not visible: by Corollary \ref{action}, the unit ball of the stable norm of $MX$ has a flat in the direction $M^{-1} \begin{pmatrix} m \\ n \end{pmatrix}$, and so do the unit ball of the stable norm of $X'$ in the direction $(m,n)$.

\end{proof}

Note that contrary to Corollary \ref{action} we did not provide any formulas for the stable norm. The reason for that is that the quantitative statements of this Corollary involve references to the Farey parents of a rational number, but the notion of Farey parents is \emph{not} preserved under the action of $M$. Thus, it is not possible to give a statement about the stable norm of $X'$ that is both \emph{quantitative} and \emph{intrinsic}, as we would need to pull things back to $X$ in order to talk about Farey parents without any ambiguity.

\subsubsection{Irrational slit tori}

Let $X'$ be a square slit torus whose slit is represented by the vector $(\beta,\alpha)$. We denote $L= \sqrt{\beta^2 + \alpha^2}$ the length of the slit of $X'$. Suppose that $X'$ is irrational, i.e $\alpha/\beta \in \RR \backslash\QQ$.
This time, it is not possible to see $X'$ as the image of a square torus with a vertical slit under the action of some matrix of $SL(2,\ZZ)$. \newline 

The idea is to approximate $X'$ with a sequence $(X_k)_{k\in \NN}$ of rational square slit tori, with a slit of length $L$ and of rational slope $p_k/q_k$, with $p_k,q_k\in \ZZ$, such that $p_k/q_k\longrightarrow \alpha/\beta$ as $k$ goes to infinity.
For $k$ big enough, one expects $X_k$ to be close to $X'$ in any reasonable topology; we claim this is true in particular in the Hausdorff-Lipschitz topology. Gutkin and Massart showed (Lemma 5 of \cite{Gut-Mas}) that the stable norm is continuous with respect to the Hausdorff-Lipschitz topology. Hence the unit ball $\Bb$ of the stable norm of $X'$ is the limit of the unit ball $\Bb_k$ of the stable norm of $X_k$, which is described by Corollary \ref{rational}. This, however, is not fully satisfying for two reasons. First, as we mentioned earlier, it is rather tedious to produce explicit formulas for the stable norm of $X_k$, let alone taking their limit when $k$ goes to infinity. Second and most importantly, we do not get a combinatorial description of the unit ball of the stable norm of $X'$. Indeed, many things can happen when taking a limit of convex sets: for instance, the regular $n$-gon becomes a circle when $n$ goes to infinity, but the combinatorial geometries of the $n$-gon and of the circle are very different. \newline 

\paragraph{Visible classes.}
In order to describe the qualitative geometry of the unit ball $\Bb$ of the stable norm of $X'$ we will proceed as in the case of rational slit tori: we find what the visible classes are, and among non-visible classes we determine which classes correspond to a strictly convex direction of $\Bb$ and which classes lie inside a flat of $\Bb$.

\begin{prop}
    Let $(m,n)\in H_1(X',\ZZ)$ be a primitive integral homology class. Recall that the slit is represented by the vector $(\beta,\alpha)$, with $\alpha/\beta \in \RR\backslash \QQ$. Then $(m,n)$ is visible if and only if $|m\alpha -n\beta| \le 1$.
\end{prop}

\begin{proof}
Note that this is the same condition that we found for rational tori in Corollary \ref{rational}, but this time we cannot take things back to a torus with a vertical slit that we already understand. \newline

By definition, the class $(m,n)$ is visible if it is possible to go in a straight line from $(0,0)$ to $(m,n)$ in the abelian covering of $X'$. In the square model of $X'$, draw a line of slope $n/m$ starting from the bottom left corner of the square, that is, the lower endpoint of the slit. Draw the parallel line passing through the upper endpoint of the slit; this line intersects the left side of the square at height $\alpha - \mathlarger{\frac{n}{m}}\beta$. These two lines bound a band containing the slit. Now, a line of slope $n/m$ in the abelian cover of $X'$ intersects the slit if and only if its projection to $X'$ enters this band at any point. Since the projection of any line of slope $n/m$ is parallel to this band, if it enters the band at any point it means it is trapped inside the band forever; in particular, it intersects the left side of the square inside of the band, i.e at height \emph{below} $\alpha - \mathlarger{\frac{n}{m}}\beta$. 

\begin{center}
    \includegraphics[scale=0.7]{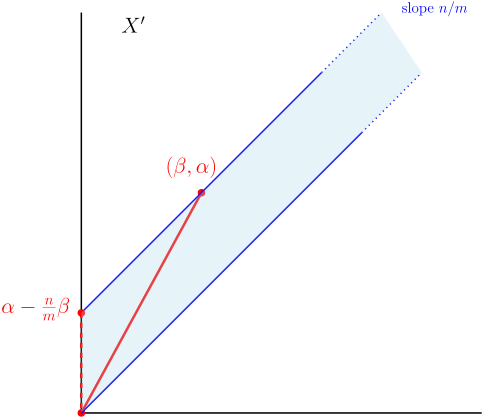}
    \captionof{figure}{The slit $(\beta,\alpha)$ casts the same shadow as a vertical slit of length $\alpha - \mathlarger{\frac{n}{m}}\beta$}
\end{center}

So the problem comes back to a vertical slit torus: more precisely, the class $(m,n)$ is visible on $X'$ if and only if it is visible on a square slit torus with a \emph{vertical} slit of length $\alpha - \mathlarger{\frac{n}{m}}\beta$. By Proposition \ref{visible}, we know this is equivalent to 
$$ |m.(\alpha - \mathlarger{\frac{n}{m}}\beta)|\le 1$$
i.e to $|m\alpha -n\beta| \le 1$.
\end{proof}

Remark that the quantity $|m\alpha -n\beta|$ is the area of the forbidden band, i.e the size of the shadow cast by the slit on $X'$ when hit by parallel light rays of slope $n/m$. Since $X'$ has area $1$, the condition $|m\alpha -n\beta|\le 1$ simply means that this shadow does not cover the entire torus, so there is still room for straight line curves representing $(m,n)$ to exist.

\begin{cor}
    Integral visible classes are isolated.
\end{cor}

\begin{proof}
    Let $(m,n)$ be an integral visible class, i.e $|m\alpha -n\beta| \le 1$. This implies $$\left |\frac{\alpha}{\beta} - \frac{n}{m} \right | \le \frac{1}{\beta m} $$

    So for $m$ fixed, there is only up to a finite number of visible classes. Moreover, since $\alpha/\beta$ is irrational, there exists a constant $c(m)$ that depends only on $m$ such that
    $$0 < c(m) \le \left |\frac{\alpha}{\beta} - \frac{n}{m} \right | \le \frac{1}{\beta m} $$
    But $1/\beta q$ goes to $0$ as $q$ goes to infinity, thus there is only a finite number of visible classes $(q,p)$ such that $| \alpha/\beta - p/q| \ge c(m)$, i.e there is only a finite number of visible classes $(q,p)$ that may be close to $(m,n)$. So $(m,n)$ is isolated.
    
\end{proof}

Now we would like to characterize visible classes on $X'$ through the approximation by rational tori $X_k$. More precisely, let $X_k$ be the square torus with a slit of length $L=\sqrt{\beta^2+\alpha^2}$ and of slope $p_k/q_k\in \QQ$.  Moreover, assume that $p_k/q_k$ is the $k$-th convergent to the continued fraction expansion of $\alpha/\beta$, i.e if $\alpha / \beta$ has (unique) infinite continued fraction expansion $[a_0,a_1,...,a_n,...]$ then $p_k/q_k= [a_0,a_1,...,a_k]$.

\begin{prop}\label{limit}
    A class $(m,n)$ is visible on $X'$ if and only if $(m,n)$ is visible on infinitely many $X_k$'s.
\end{prop}

Note that since all considered slit tori are generated by a square we have a common natural basis for their homology, so it makes sense to pick a homology class $(m,n)$ and to discuss and compare its properties "on $X'$" and "on $X_k$".

\begin{proof}
    Suppose that $(m,n)$ is visible on $X'$, by definition we have $|m\alpha - n\beta| \le 1$. \\ First, assume $|m\alpha - n\beta| < 1$. The quantity $|m\alpha - n\beta|$ is the area of the parallelogram generated by the vector $(m,n)$ and the slit of $X'$, represented by the vector $(\beta,\alpha)$. On $X_k$, the slit has length $L$ and slope $p_k/q_k$, thus it is represented by the vector $\mathlarger{\frac{L}{q_k^2+p_k^2}}.(q_k,p_k)$. Then by Corollary \ref{rational}, the class $(m,n)$ is visible on $X_k$ if and only if 
    $$ \frac{L}{q_k^2+p_k^2} |m p_k - nq_k| \le 1$$
    that it to say if and only if the parallelogram generated by the vector $(m,n)$ and the slit of $X_k$ has area less than $1$. But by construction, the slit of $X_k$ goes to $(\beta, \alpha)$ when $k$ goes to infinity, so this parallelogram goes to the parallelogram generated by $(m,n)$ and $(\beta, \alpha)$ of area strictly less than $1$. Thus for $k$ big enough, the parallelogram generated by the vector $(m,n)$ and the slit of $X_k$ must have area less than $1$, i.e $(m,n)$ is visible on $X_k$. \\
    Now, what happens if $|m\alpha - n\beta| = 1$? Note that necessarily $\alpha$ and $\beta$ are linearly dependant over $\QQ$, which is not true in general. In that case, there is a slight precaution to take: since this time the limit area of the parallelograms is not strictly less than $1$, we cannot conclude that all parallelograms of the sequence have area less than $1$ for $k$ big enough. Instead, we claim that half of them have area less than $1$. Indeed, since the value of a continued fraction lies between any two of its consecutive convergents, for all $k$ we have $p_k/q_k \le \alpha /\beta \le p_{k+1}/q_{k+1}$ or $ p_{k+1}/q_{k+1} \le \alpha /\beta \le p_k/q_k$. Let $k_0\in \NN$ such that $p_{k_0}/q_{k_0}$ lies between $n/m$ and $\alpha/\beta$. Then the parallelogram generated by $(m,n)$ and the slit of $X_{k_0}$ has area \emph{strictly} less than $1$, so $(m,n)$ is visible on $X_{k_0}$. Moreover, this is also true for any $k\ge k_0$ with the same parity as $k_0$. Thus $(m,n)$ is visible on infinitely many of the $X_k$'s. \newline

    Conversely if $(m,n)$ is visible on infinitely many $X_k$'s, say on a subsequence $X_{\varphi(k)}$, then the parallelograms generated by $(m,n)$ and the slit of $X_{\varphi(k)}$ have area less than $1$ and so does their limit, i.e $|m\alpha - n\beta| \le 1$ and $(m,n)$ is visible on $X'$.
\end{proof}

We now understand the visible classes on $X'$; what about the non-visible ones? Obviously, by Proposition \ref{limit}, if a class $(m,n)$ is not visible on $X'$ this is also the case on infinitely many $X_k$'s, and conversely. As in the case of rational tori the notion of Farey parents might not be relevant anymore, but one thing key thing remains: the visible neighbours.

\begin{lemma}
    Let $(m,n)$ be a homology class on $X'$, and assume that $(m,n)$ is not visible. Then it has at most two visible neighbours. More precisely, there exists at most two rationals $p/q$ and $p'/q'$ neighbours to $n/m$ in some Farey sequence (possibly in a different order) such that the homology classes $(q,p)$ and $(q',p')$ are visible on $X'$.
\end{lemma}

\begin{proof}
Let $X$ be a square torus with a vertical slit. On $X$, a non-visible class has up to two visible neighbours, which are its Farey parents. Since rational tori are obtained by applying a $SL(2,\ZZ)$ matrix transformation to a torus like $X$, this means that on any rational torus a non-visible class has at most two visible neighbours. Since the $X_k$ are rational, for any $k\in \NN$ the class $(m,n)$ has up to two visible neighbours. By Proposition \ref{limit}, if $(m,n)$ had three or more visible neighbours on $X'$ they should be also visible on the $X_k$ for $k$ big enough, which is impossible since there are at most two visible neighbours to a non-visible class on $X_k$. Hence the result.
\end{proof}

By analogy with the case of the vertical slit, let us consider separately the cases of classes with and without visible neighbours. This is analogous to the case of the torus with a vertical slit, where we separated classes with at least one visible Farey parent from classes without any visible Farey parent.

\paragraph{Classes with visible neighbours.}

Let $(m,n)$ be a non-visible integral class on $X'$ and assume that it has one visible neighbour, denoted $(q,p)$. By Proposition \ref{limit}, up to taking a sub-sequence, for $k$ big enough the situation is the same on $X_k$, that is to say the class $(m,n)$ is not visible on $X_k$ and the class $(q,p)$ is visible. By Corollary \ref{rational}, we already know that $(m,n)$ is a strictly convex direction of the unit ball of the stable norm of $X_k$, and now we would like to find a formula for its stable norm. Taking the limit when $k$ goes to infinity we will finally obtain a formula for the stable norm of $(m,n)$ in $X'$. \newline

Since $X_k$ is a rational torus, there exists a square torus $T_k$ with a vertical slit and a matrix $M_k \in SL(2,\ZZ)$ such that $X_k = M_k . T_k$. With the same reasoning as in Corollary \ref{action} and Corollary \ref{rational}, there exists a class $(\mu_k,\nu_k)$ on $T_k$ that is sent to $(m,n)$ by the action of $M_k$. Let $a_k/b_k$ and $c_k/d_k$ be the Farey parents of $\nu_k/\mu_k$, and let $(b_k,a_k)$ and $(d_k,c_k)$ be the associated homology classes on $T_k$. In order to express the stable norm of $(m,n)$ on $X_k$, what we need is to find the images of $(b_k,a_k)$ and $(d_k,c_k)$ under the action of $M_k$. Since $(m,n)$ has a visible neighbour $(q,p)$, we know the class $(q,p)$ is the image of one of them under the action of $M_k$. Indeed, the $SL(2,\ZZ)$-action on slit tori preserves the properties of being neighbours and of being visible, and the only visible classes neighbours to $(\mu_k,\nu_k)$ on $T_k$ are its Farey parents. So we have found the image of one of the Farey parents of $(\mu_k,\nu_k)$; what about the second one? \newline

If $(m,n)$ has a second visible neighbour $(q',p')$ on $X'$ (or equivalently on $X_k$ for $k$ big enough) then we are done, as it has to be the image of the second Farey parent of $(\mu_k,\nu_k)$, and by Corollary \ref{action} for all $k$ big enough we have

\begin{equation} \label{pvisible}  
||(m,n)||_{X_k} = ||(q,p) \pm (\beta_k,\alpha_k)||_2 + || (q',p') \mp (\beta_k,\alpha_k)||_2
\end{equation}

where $(\beta_k,\alpha_k)$ denotes the slit of $X_k$.
The term with the "$+$" sign corresponds to the one containing the image of the smallest Farey parent of $(\mu_k,\nu_k)$, as in Theorem \ref{main thm}. \newline 
 If $(m,n)$ has exactly one visible neighbour, the above formula still holds, although it is a bit trickier to establish. There are exactly two rationals, denoted $p'/q'$ and $p''/q''$, that are neighbours to both $n/m$ and $p/q$ is some Farey sequence: one of them is their Farey child and the other is their common Farey parent. Thus, one of the classes $(q',p')$ and $(q'',p'')$ has to be the image of the second Farey parent of $(\mu_k,\nu_k)$ under the action of $M_k$. Say it is $(q',p')$. Now, up to taking a sub-sequence we can assume that $(q',p')$ is always the image of the second Farey parent of $(\mu_k,\nu_k)$ under the action of $M_k$. Hence in this case, the formula  \ref{pvisible} still holds for all $k$ big enough. \\

 Thus, since $(\beta_k,\alpha_k) \longrightarrow (\beta,\alpha)$ when $k\rightarrow\infty$, by taking the limit of the expression \ref{pvisible} as $k$ goes to infinity we finally get

$$||(m,n)||_{X'} = ||(q,p) \pm (\beta,\alpha)||_2 + || (q',p') \mp (\beta,\alpha)||_2$$

The same argument as in Proposition  \ref{2 parent} and Proposition \ref{1 parent} then shows that the unit ball of the stable norm of $X'$ is strictly convex in the direction $(m,n)$.

\paragraph{Classes with no visible neighbour.}

Let $(m,n)$ be a non-visible integral class on $X'$, and assume it has no visible neighbour. By Proposition \ref{limit}, this is also the case on $X_k$ for $k$ big enough. By Corollary \ref{rational}, this means that the unit ball of the stable norm of $X_k$ has a flat in the direction $(m,n)$. In other words, there exist two integral classes $(q,p)$ and $(q',p')$ such that the point $\mathlarger{\frac{(m,n)}{||(m,n)||_{X_k}}}$ lies inside of the segment $I = \left [\mathlarger{\frac{(q,p)}{||(q,p)||_{X_k}}} , \mathlarger{\frac{(q',p')}{||(q',p')||_{X_k}}} \right ]$. We can assume that no visible class appears inside of $I$ as $k$ goes to infinity. Indeed, if a class inside of this segment became visible as $k$ increases we would just replace $I$ by a shorter segment. On the worst case, we would have to shrink $I$ only a finite number of times: if infinitely many visible classes would continue to appear as we take shorter and shorter segments this would give us a sequence of visible classes converging to $(m,n)$. But the only accumulation point of visible classes is the \emph{irrational} class $(\beta,\alpha)$, so this is absurd. \newline

Thus for $k$ big enough we have 
$$||(m,n)||_{X_k} = ||(q,p)||_{X_k}+ ||(q',p')||_{X_k}$$
By taking the limit as $k$ goes to infinity, we get
$$||(m,n)||_{X'} = ||(q,p)||_{X'}+ ||(q',p')||_{X'}$$

In particular the point $\mathlarger{\frac{(m,n)}{||(m,n)||_{X'}}}$ lies inside of the segment $ \left [\mathlarger{\frac{(q,p)}{||(q,p)||_{X'}}}, \mathlarger{\frac{(q',p')}{||(q',p')||_{X'}}} \right ]$, so the unit ball of the stable norm of $X'$ has a flat in the direction $(m,n)$. \newline \newline 

We have shown

\begin{thm}
    Let $X'$ be a square torus with a slit $(\beta,\alpha)$ of any direction. The unit ball of the stable norm of $X'$ has a vertex in the direction $(m,n)\in H_1(X',\ZZ)$ if and only if either
    \begin{itemize}\renewcommand{\labelitemi}{$\bullet$}
    \item $(m,n)$ is visible, i.e $|m\alpha - n\beta| \le 1$. \\ \\
    or 
    \item $(m,n)$ is not visible but is neighbour to at least one visible class.
    \end{itemize}
    The unit ball of the stable norm of $X'$ has a flat in any other direction of $H_1(X',\RR)$.
\end{thm}

\section{Glueing slit tori}

As we now fully understand the stable norm of slit tori we can move on to the second objective of this paper, which is glueing slit tori together to obtain flat surfaces on which we can compute the stable norm. \newline \newline

Let $X_1,...,X_n$ be $n$ identical copies of the square slit torus of area $1$ with a vertical slit of length $\rho$. We will glue them along flat cylinders and obtain a flat surface of genus $n$. In order to simplify the notations we will assume that $n=2$; the proof is identical in the general case. \newline

Let $\Sigma$ be the closed genus $2$ surface obtained by glueing a flat cylinder $C$ of width $w$ along the slits of $X_1$ and $X_2$. By construction, the surface $\Sigma$ is endowed with a flat metric, obtained by glueing the flat metrics of the two slit tori and of the cylinder. By Riemann's uniformization theorem, we know that this metric cannot be smooth; indeed, a closer inspection shows the metric has four conical singularities of angle $3 \pi$ located at each endpoint of the slits, and is smooth elsewhere. One can check that $\Sigma$ is a half-translation surface; see for instance Zorich's survey \cite{Zorich} for more on the geometry of these surfaces. We have

$$ H_1(\Sigma,\ZZ) = H_1(X_1,\ZZ)\oplus H_1(X_2,\ZZ) \simeq \ZZ^4$$

so if $(e_i,f_i)$ is a basis of $H_1(X_i,\ZZ)$, we have $H_1(\Sigma,\ZZ) = \langle e_1,f_1,e_2,f_2 \rangle$. \newline

We denote $\StS{.}$ the stable norm of $\Sigma$, and $||.||_{X_i}$ the stable norm of the slit tori $X_i$.
\begin{lemma}\label{restriction}
    The stable norm of $\Sigma$ coincides with the stable norm of $X_i$ on the plane $Span(e_i,f_i)$.
\end{lemma}

\begin{proof}
    Assume for instance $i=1$, and let $h=(m,n,0,0)\in H_1(\Sigma,\ZZ) \cap Span(e_1,f_1)$ be an integral homology class and $\gamma$ be a minimizing curve representing $h$. We will show that $\gamma$ is a closed curve contained in $X_1$. \\ 
    Let $\gamma_2 := \gamma \cap X_2$ be the part of $\gamma$ (possibly empty) inside of $X_2$; for simplicity let us assume that $\gamma_2$ is an arc. Close this arc with a path $\mu_2 \subset X_2$ that goes around the slit in the shortest way to obtain a closed curve $\tilde{\gamma}_2$ in $X_2$. Note that $\mu_2$ has length at most $\rho$ the length of the slit. Let $\tilde{\gamma}$ be the closed curve obtained by closing $\gamma - \gamma_2$ with the path $\mu_2$.
    
    \begin{center}
    \includegraphics[scale=0.8]{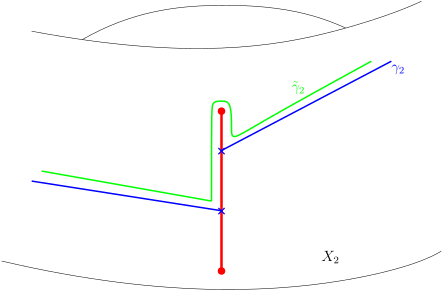}
    \captionof{figure}{Closing $\gamma_2$ by going around the slit.}
\end{center}
By construction, $[\gamma]= [\tilde{\gamma}] + [\tilde{\gamma}_2] = (m,n,0,0)$. As $\tilde{\gamma}_2$ has intersection number zero with the canonical representatives of the homology classes $(1,0,0,0)$ and $(0,1,0,0)$, it represents a homology class whose first two coordinates are zero. Similarly, the curve $\tilde{\gamma}$ represents a homology class whose last two coordinates are zero. Since the sum of those two classes is $(m,n,0,0)$, we get  $[\tilde{\gamma}_2]=0$ so the curve $\tilde{\gamma}$ represents the same homology class as $\gamma$. We have
$$l(\tilde{\gamma}) = l(\gamma) + l(\mu_2) - l(\gamma_2)$$

Assume that $\gamma_2 := \gamma \cap X_2$ is non-trivial, i.e $l(\gamma_2)>0$. Necessarily the closed curve $\tilde{\gamma}_2$ has non-trivial homotopy class in the slit torus $X_2$: otherwise we could shrink $\tilde{\gamma}_2$ (and thus $\gamma_2$) onto a point and reduce the length of $\gamma$, contradicting its minimality. Thus \begin{itemize} 
\item either $\tilde{\gamma}_2$ has to loop around the torus $X_2$ at least once, and since $X_2$ is obtained from a square of side of length $1$ we have $l(\gamma_2)\ge 1$ \item either $\tilde{\gamma}_2$ is homotopic to a circle going around the slit so $l(\tilde{\gamma}_2) \ge 2\rho$, and since $l(\tilde{\gamma}_2) = l(\gamma_2) + l(\mu_2)$ and $l(\mu_2) \le \rho$ we have $l(\gamma_2) \ge \rho$.
\end{itemize}

In both cases we have $l(\gamma_2) \ge l(\mu_2)$. Hence $l(\tilde{\gamma}) = l(\gamma) + l(\mu_2) - l(\gamma_2) < l(\gamma)$, which contradicts the minimality of $\gamma$. Thus $l(\gamma_2)=0$ and the minimizing curve $\gamma$ does not enter $X_2$. \\ 

We repeat the previous process: let $\gamma_1 := \gamma \cap X_1$ and $\gamma_C := \gamma \cap C$. We get a closed curve $\tilde{\gamma_1}$ by closing $\gamma_1$ with the path $\mu_1 \subset X_1$ going around the slit in the shortest way. Again, $l(\mu_1) \le \rho$. Set $\tilde{\gamma_C} = \gamma_C \cup \mu_1$. It is a closed curve contained in $C$ with trivial homology class, hence $\tilde{\gamma}_1$ represents the same homology class as $\gamma$. If $l(\gamma_C) >0$ clearly we have $l(\gamma_C) > l(\mu_1)$ so $$l(\gamma_1) + l(\mu_1) =l(\tilde{\gamma}_1) < l(\gamma) = l(\gamma_1) + l(\gamma_C)$$
which contradicts the minimality of $\gamma$. 

\begin{center}\label{excursion}
    \includegraphics[scale=0.75]{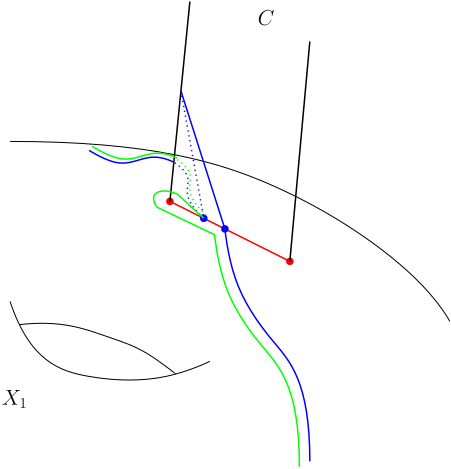}
    \captionof{figure}{The green path is shorter than the blue one.}
\end{center} 

Thus $l(\gamma_C)=0$, i.e $\gamma$ is a closed curve that is contained in $X_1$, and by definition of the stable norm $\gamma$ is minimizing for the homology class $(m,n)$ in $X_1$. We have shown $\| (m,n,0,0)\|_{\Sigma} = \| (m,n) \|_{X_1}$.
\end{proof}

We can now completely compute the stable norm of $\Sigma$. Just like we did for the square slit torus, we look for minimizing curves for $h$ on $\Sigma$ by separating all of its representatives into two families, inside of each we can easily minimize the length, then compare the results. This time, we separate the curves according to whether or not they cross the cylinder.

\begin{prop}
If the cylinder $C$ is long enough, more precisely if $w>\rho$ with $\rho$ the length of the slits, then for all $(m,n,p,q)\in H_1(\Sigma,\ZZ)$ we have $$\StS{(m,n,p,q)} = ||(m,n)||_{X_1}+ ||(p,q)||_{X_2}$$
\end{prop}

\begin{proof}
Let $h=(m,n,p,q)\in H_1(\Sigma,\ZZ)$ be an integral class in $\Sigma$. If $(m,n)=(0,0)$ or $(p,q)=(0,0)$, we are in the setting of Lemma \ref{restriction} and there is nothing left to prove, so let us assume $(m,n)\ne (0,0)$ and $(p,q)\ne (0,0)$. Let $\gamma$ be a representative of $h$. Note that $\gamma$ is not necessarily connected. We consider two distinct cases, depending on whether or not $\gamma$ crosses the cylinder. \newline \newline

First, consider the case where $\gamma$ does not cross the cylinder: more precisely, let us assume that there does not exist an arc of $\gamma$ whose endpoints lie in $X_1$ and $X_2$ respectively. By the same argument as in the proof of Lemma \ref{restriction} it is not efficient for $\gamma$ to venture into the cylinder without fully crossing it, so for $\gamma$ to be as short as possible necessarily $\gamma \cap C = \emptyset$. Thus $\gamma = \gamma_1 \cup \gamma_2$, with $\gamma_1=\gamma\cap X_1$ (resp. $\gamma_2=\gamma\cap X_2$) a multicurve contained in $X_1$ (resp. $X_2$) representing the homology class $(m,n,0,0)$ (resp. $(0,0,p,q)$). The length of $\gamma$ is minimal if and only if both $\gamma_1$ and $\gamma_2$ are minimizing in their homology class, so by Lemma \ref{restriction} the minimal length for $\gamma$ is $||(m,n)||_{X_1}+ ||(p,q)||_{X_2}$. \newline \newline

Second, assume that $\gamma$ crosses the cylinder at least once. For simplicity assume that it crosses the cylinder exactly once in each direction; the same argument will hold if $\gamma$ crosses the cylinder more than once. Denote $\gamma_C := \gamma \cap C$ the part of $\gamma$ that lies in the cylinder, and $\gamma_i = \gamma \cap X_i$ the part of $\gamma$ inside of $X_i$. We have $$l(\gamma)=l(\gamma_C)+l(\gamma_1)+l(\gamma_2)$$

This time the components in $X_i$ are not a union of closed curves, as we cut when $\gamma$ went into the cylinders: hence $\gamma_1$ and $\gamma_2$ are composed of an arc and possibly some closed curves. Now close the arc in $\gamma_1$ by adding the shortest path between its endpoints (the points where $\gamma$ exits and re-enters $X_1$) that follows the edges of the slit (see Figure \ref{excursion}). Note that this path has length at most $\rho$. We obtain a union $\tilde{\gamma}_1$ of closed curves in $X_1$, with $l(\tilde{\gamma}_1) \le l(\gamma_1) + \rho$. Moreover, $\tilde{\gamma}_1$ represents the homology class $(m,n,0,0)$, so $||(m,n)||_{X_1} \le l(\tilde{\gamma}_1)$ and finally $$l(\gamma_1) \ge ||(m,n)||_{X_1} - \rho$$
The same construction in $X_2$ yields $$l(\gamma_2) \ge ||(p,q)||_{X_2} - \rho$$

Finally, since $\gamma$ crosses the cylinder, there is an arc going from $X_1$ to $X_2$ through the cylinder; but this arc is part of a closed curve so $\gamma$ has to go back to $X_1$ eventually. Hence $\gamma_C$ in a union of two arcs whose endpoints lie in $X_1$ and $X_2$, so $$l(\gamma_C) \ge 2w$$ where $w$ is the width of the cylinder.
Putting everything together we finally get that in this case
$$l(\gamma) \ge ||(m,n)||_{X_1} + ||(p,q)||_{X_2} + 2(w - \rho) $$
so if $w>\rho$ this is strictly longer than in the case of curves that do not cross the cylinder. Hence the shortest representative of $(m,n,p,q)$ in $\Sigma$ is a multicurve that does not encounter the cylinder, and it has length $||(m,n)||_{X_1} + ||(p,q)||_{X_2}$.

\end{proof}

Thus, if $(m,n)\ne (0,0)$ and $(p,q)\ne (0,0)$, the point $\mathlarger{\frac{(m,n,p,q)}{\StS{(m,n,p,q)}}}$ of the unit sphere of the stable norm of $\Sigma$ in the direction $(m,n,p,q)$ lies in the interior of the segment of endpoints $\mathlarger{\frac{(m,n,0,0)}{\StS{(m,n,0,0)}}}$ and $\mathlarger{\frac{(0,0,p,q)}{\StS{(0,0,p,q)}}}$. Hence, if the unit ball $\Bb$ of the stable norm of $\Sigma$ has any vertices they must lie in the homology planes associated to the slit tori $X_1$ and $X_2$, that is $Span(e_1,f_1)$ and $Span(e_2,f_2)$. With the same argument as in Theorem \ref{main thm}, one can check that the vertices of the stable norm of $\Sigma$ correspond exactly to the vertices of the stable norm of $X_1$ and $X_2$. More precisely, $(m,n,0,0)$ (resp. $(0,0,p,q)$) is the direction of a vertex of the stable norm of $\Sigma$ if and only if $(m,n)$ (resp. $(p,q)$) is the direction of a vertex of the stable norm of $X_1$ (resp. $X_2$).

We have shown

\begin{thm}\label{glue}
    The unit ball of the stable norm of $\Sigma$ is the convex hull of the set of its vertices, that is to say the convex hull of the set
    $$ \left \{ \frac{(m,n,0,0)}{\StS{(m,n,0,0}} \text{ with } (m,n)\in V(X_1)  \right \} \cup \left\{\frac{(0,0,p,q)}{\StS{(0,0,p,q)}} \text{ with } (p,q)\in V(X_2)\right\} $$ where $V(X_i)$ is the set of directions of vertices of the stable norm of $X_i$.
\end{thm}

In particular, the unit ball of the stable norm of $\Sigma$ has flats of dimension $3$, which is maximal since $H_1(\Sigma,\RR)$ has dimension $4$. More generally, glueing $n$ slit tori with the same construction gives us a half-translation surface of genus $n$ whose flats of the stable norm are of maximal dimension, that is $2n-1$. \newline

\section{Counting simple homology classes}

A classical problem related to the stable norm is to find the homology classes that are minimized by \emph{nice} representatives. For instance, we can ask for the representatives to be simple:

\begin{definition}
    Let $S$ be a (possibly singular) Riemannian compact surface. A homology class $h\in H_1(S,\ZZ)$ is simple if it is minimized by a simple closed curve.
\end{definition}

One natural question is a counting one: how many are there simple homology classes on a given surface? In some sense, this is orthogonal to the classical curve counting problem, namely: given a fixed topological type (or a homology class), how many are there curves of this topological type of length less than a positive constant $x$ ? The question was to count short curves representing the same homology class, and now we would like to count the homology classes admitting a short representative. \newline 
More precisely, given a surface $S$ and a positive number $x$, we are interested in the following question: how many are there simple homology classes on $S$ whose stable norm does not exceed $x$?
It is known (see \cite{Gut-Mas}) that on translation surfaces this quantity grows quadratically in $x$. Since half-translation surfaces are so close to translation surfaces, one would expect the number of simple homology classes to be quadratic in $x$ as well on half-translation surfaces. But this is not true, and the aim of this section is to provide a counter-example. \newline \newline

Let $X$ be a square torus with a vertical slit of length $\rho$.
The surface we are interested in is the genus $2$ half-translation surface $\Sigma$ from the previous section, obtained by glueing two copies of $X$ along a flat cylinder. By Theorem \ref{glue}, since the unit ball of the stable norm of $\Sigma$ is the convex hull of two copies of the unit ball of the stable norm of $X$, counting the number of simple classes on $\Sigma$ is equivalent to counting twice the number of simple classes on $X$. Thus we will count these classes on $X$ as it is the building block of $\Sigma$. 

\begin{thm}
    Let $p(x)= \#  \{ h \in H_1(X,\ZZ),$ with $h$ simple and $\St{h}\le x \}$. Then
    $$p(x) \sim 4 \left(\sum_{b=1}^{\lfloor 1/\rho \rfloor} \frac{\varphi(b)}{b}\right )  x \ln x + O(x)$$
    where $\varphi$ is Euler's totient function and $\lfloor . \rfloor$ denotes the integral part.
\end{thm}

We use classical counting methods from analytic number theory, see for instance \cite{Apostol} or \cite{Tenenbaum}.

\begin{proof}
Let $L :=  \{ \St{h}$ with $h\in H_1(X,\ZZ)$ simple$ \}$ be the length spectrum of $X$, where lengths are repeated with multiplicity. This is a countable set of positive numbers, and we order it by increasing order: $L = \left \{ l_1 \le l_2 \le ... \le l_n \le ... \right \}$. For any complex number $s$ with $\Re(s) > 1$, let 
$$F(s):= \sum_{n=1}^\infty \frac{1}{l_n^s}$$
This is a generalized Dirichlet series of the form $F(s) = \sum\limits_{n=1}^\infty a_n e^{-\lambda_n s}$ where for all $ n, a_n = 1$ and $\lambda_n = \ln l_n$. Applying Perron formula,  for any $\sigma >1$ and $y\in ]\lambda_n, \lambda_{n+1}[$ we get

$$ \sum_{k=1}^n a_k = \frac{1}{2i\pi}\int_{\sigma-i\infty}^{\sigma+i\infty} F(s) \frac{e^{sy}}{s}ds$$

The left hand term is the number of $\lambda_k$ such that $\lambda_k \le y$, or equivalently this is the number of $l_k$ such that $l_k\le e^y$. By setting $x:= e^y$, we get

$$p(x) = \frac{1}{2i\pi}\int_{\sigma-i\infty}^{\sigma+i\infty} F(s) \frac{x^s}{s}ds$$

so all we have to do is to compute the right hand side integral. To do this, we need to find a more practical expression of $F(s)$. \newline \newline

How are simple classes distributed in $H_1(X,\ZZ)$? These classes correspond to the directions of the vertices of the unit ball of the stable norm of $X$, i.e they are either visible classes or the Farey child of some visible class. Each visible class has infinitely many Farey children, that are all simple classes. Thus we can rearrange the sum in $F(s)$:

$$ F(s) = \sum_{ (b,a) \text{ visible }} F_{b,a}(s)   - E(s)$$
where $$F_{b,a}(s) = \frac{1}{\St{(b,a)}^s} + \sum\limits_{(q,p) \text{ child of } (b,a)} \frac{1}{\St{(q,p)}^s}$$ and $E(s)$ is an excess correcting term, as we may count some classes multiple  times. \newline \newline

Let us find a more precise expression for $F_{b,a}(s)$. The successive Farey children of $(b,a)$ are of the form $(kb+\beta, ka+\alpha)$ and $(kb+\delta, ka+\gamma)$ where $(\beta,\alpha)$ and $(\delta,\gamma)$ are two other visible classes such that $\alpha /\beta , a/b$ and $\gamma/\delta$ are consecutive neighbours in some Farey sequence, and $k\in\NN$. 
Since $(b,a)$ is visible, we have
$$\St{(b,a)} = \|(b,a)\|_2$$
For $k$ big enough, $(kb+\beta, ka+\alpha)$ is not a visible class so by Theorem \ref{main thm}
$$\St{(kb+\beta, ka+\alpha)} = \sqrt{((k-1)b + \beta)^2 + ((k-1)a+\alpha \pm \rho)^2} + \sqrt{b^2+(a\mp \rho)^2}$$
Taking the Taylor series expansion of this expression, we get
$$\St{(kb+\beta, ka+\alpha)} = k \| (b,a)  \|_2 + \frac{c}{k} + O\left (\frac{1}{k^2}\right )$$
where $c$ is a constant independent of $k$. Thus,

$$\frac{1}{\St{(kb+\beta, ka+\alpha)}^s} = \frac{1}{k^s \|(b,a)\|_2^s} - \frac{sc}{k^s} + r_{1,k}(s)$$

where $|r_{1,k}(s) / k^{s+2}|\longrightarrow 0$ when $k$ goes to infinity.

Similarly, we have
$$\frac{1}{\St{(kb+\delta, ka+\gamma)}^s} = \frac{1}{k^s \|(b,a)\|_2^s} + \frac{sc}{k^s} + r_{2,k}(s)$$
where $|r_{2,k}(s) / k^{s+2}|\longrightarrow 0$ when $k$ goes to infinity.

Injecting it back into $F_{b,a}(s)$, we obtain

$$F_{b,a}(s) = \frac{1}{\|{(b,a)}\|_2^s} + 2\sum_{k=1}^\infty \frac{1}{k^s \| (b,a) \|_2 ^s} + \sum_{k=1}^\infty (r_{1,k}(s) + r_{2,k}(s))$$

so

$$ F_{b,a}(s) = \frac{1}{\|(b,a)\|_2^s } \left( 2\zeta(s) +1\right )  + R_{b,a}(s)$$
where $$\zeta(s) = \sum_{k=1}^\infty \frac{1}{k^s}$$ is the Riemann zeta function, and $|R_{b,a}(s)/\zeta(s+2)|$ is bounded for $\Re(s)>1$. \newline \newline

Summing over all visible classes, we have
$$\sum_{(b,a) \text{ visible }} F_{b,a}(s) =  \left( 2\zeta(s) +1\right ) \sum_{(b,a) \text{ visible }}  \frac{1}{\|(b,a)\|_2^s }  +   \sum_{(b,a) \text{ visible }} R_{b,a}(s)$$
Since the terms of the Farey sequence between two integers $n$ and $n+1$ are simply the terms of the Farey sequence between $0$ and $1$ translated by $n$, we can rewrite the first right hand side sum as follows:

$$ \sum_{(b,a) \text{ visible}}  \frac{1}{\|(b,a)\|_2^s } = \sum_{\substack{(b,a) \text{ visible } \\ 0<a/b<1}} \sum_{k\in \ZZ}  \frac{1}{\|(b,a+kb)\|_2^s } = 2 \sum_{\substack{(b,a) \text{ visible } \\ 0<a/b<1}} \sum_{k\ge 0}  \frac{1}{\|(b,a+kb)\|_2^s }$$

When $k$ goes to infinity, we have 

$$\frac{1}{\|(b,a+kb)\|_2^s } = \frac{1}{bk^s} + O\left ( \frac{1}{k^{s+2}} \right )$$

so 

$$ \sum_{(b,a) \text{ visible }}  \frac{1}{\|(b,a)\|_2^s } = 2(2\zeta(s) +1) \left (\sum_{\substack{(b,a) \text{ visible } \\ 0<a/b<1}} \frac{\zeta(s)}{b} \right ) + R_1(s)$$

where $|R_1(s)/\zeta(s+2)|$ is bounded for $\Re(s) >1$. Thus we have

$$ \sum_{(b,a) \text{ visible }} F_{b,a}(s) = 4 \left ( \sum_{\substack{ b \le \lfloor 1/\rho \rfloor \\ gcd(b,a)=1}} \frac{1}{b} \right ) \zeta^2(s) + c' \zeta(s) + R_2(s) $$
where $c'$ is a constant and $|R_2(s) / \zeta^2(s+2)|$ is bounded for $\Re(s)>1$.

Finally, we have 

$$ \sum_{(b,a) \text{ visible}} F_{b,a}(s) = 4 \left ( \sum_{b=1}^{\lfloor 1/\rho \rfloor}\frac{\varphi(b)}{b} \right ) \zeta^2(s) + c' \zeta(s) + R_2(s) $$ 
where $\varphi$ is Euler's totient function. \newline

Did we count classes multiple times? Yes we did: any visible class $(m,n)$ has been counted as many times as it has Farey ancestors. Indeed, it has been counted as a Farey child of each of its ancestors, and since the first coordinate of the Farey parents are strictly less than the first coordinate of this class this means $(m,n)$ has been counted at most $m$ times. Now, the first coordinate of visible classes is bounded above by $\lfloor 1/\rho \rfloor$, so any visible class that we counted multiple times has been counted at most $\lfloor 1/\rho \rfloor$ times. Moreover, the first Farey child of two visible classes has been counted twice, once for each of its visible parents. 
Thus, even if we can't give a precise expression for $E(s)$, we deduce that $E(s)$ is dominated by 
$$\lfloor 1/\rho \rfloor \sum\limits_{(q,p) \text{ with both visible parents}} \frac{1}{\| (q,p) \|_X^s}$$
This in turn is dominated by 
$$ c'' \sum\limits_{(b,a) \text{visible}} \frac{1}{\| (b,a) \|_X^s} $$
where $c''$ is a positive constant. As we have seen earlier, this quantity is equal to $c''\zeta(s)$ plus a remainder that is small when $\Re(s)\longrightarrow 1$.
\newline \newline

Putting all the formulas together, we get

$$F(s) = 4 \left ( \sum_{b=1}^{\lfloor 1/\rho \rfloor}\frac{\varphi(b)}{b} \right ) \zeta^2(s) + C \zeta(s) + R(s)$$

where $C$ is a non-zero constant and $|R(s) / \zeta(s) |$ bounded for $\Re(s)> 1$. Recall that by Perron formula for any $\sigma >1$ we have

$$p(x) = \frac{1}{2i\pi}\int_{\sigma-i\infty}^{\sigma+i\infty} F(s) \frac{x^s}{s}ds$$

Replacing $F(s)$ in the integral by the above expression, we can compute the terms separately.

 \begin{itemize}\renewcommand{\labelitemi}{$\bullet$}
\item By Perron formula, 
$$ \frac{1}{2i\pi}\int_{\sigma-i\infty}^{\sigma+i\infty} \zeta^2(s) \frac{x^s}{s}ds = \sum_{n\le x} \tau(n) $$
where $\tau(n)=\#  \{ d$ such that $d$ divides $n \}$. This is due to a classical property of Riemann zeta function: the Dirichlet series expansion of $\zeta^2(s)$ is known to be $\sum\limits_{n\ge 1} \tau(n) n^{-s}$. Another classical number theory computation yields
$$\sum_{n\le x} \tau(n) \sim x \ln x + O(x) $$

\item Again, by Perron formula, since $\zeta(s) = \sum\limits_{n\ge 1} n^{-s}$ we have 
$$ \frac{1}{2i\pi}\int_{\sigma-i\infty}^{\sigma+i\infty} \zeta(s) \frac{x^s}{s}ds = \sum_{n\le x} 1 = O(x) $$

\item We cannot compute explicitly the term containing the remainder $R(s)$. However, since $R(s)$ is small compared to $\zeta(s)$ when $\Re(s)$ goes to $1$, and since the Perron formula holds for \emph{any} $\sigma > 1$, by making $\sigma$ go to $1$ we are assured that the integral term containing $R(s)$ is small compared to the other integral terms. In particular, this integral term is at most linear in $x$, so for $\sigma$ close enough to $1$
$$ \frac{1}{2i\pi}\int_{\sigma-i\infty}^{\sigma+i\infty} R(s) \frac{x^s}{s}ds = O(x) $$

 \end{itemize}

Summing these three terms with the right multiplicative constants we finally get

$$p(x) \sim 4 \left(\sum_{b=1}^{\lfloor 1/\rho \rfloor} \frac{\varphi(b)}{b}\right )  x \ln x + O(x)$$

\end{proof}

\begin{cor}
    The number of simple homology classes on the half-translation surface $\Sigma$ whose stable norm does not exceed $x$ is asymptotic to $$ 8  \left(\sum_{b=1}^{\lfloor 1/\rho \rfloor} \frac{\varphi(b)}{b}\right )  x \ln x + O(x)$$
    More generally, by glueing $n$ copies of the slit torus $X$ along flat cylinders, one can construct a genus $n$ half-translation surface whose number of simple homology classes of stable norm less than $x$ grows as
    $$ 4n  \left(\sum_{b=1}^{\lfloor 1/\rho \rfloor} \frac{\varphi(b)}{b}\right )  x \ln x + O(x)$$
\end{cor}

{\bf Remark.} The previous Corollary is a counterexample to the statement on singular surfaces of Proposition 3 in \cite{Gut-Mas}. Indeed, this proposition provides a quadratic lower bound on the number of simple homology classes on surfaces with conical singularities, such as half-translation surfaces. The quadratic lower bound is obtained first on non-singular surfaces and extended to singular surfaces by taking a limit on the metric: but this does not work, as the geodesics do not transform continuously when deforming the metric. More precisely, the argument that fails here is the bound on the number of connected components of minimizing curves: on a non-singular surface this number is bounded above by the genus of the surface, but on our genus $2$ surface $\Sigma$ a minimizing curve can have up to $4$ connected components. \\ \\

 To conclude we would like to comment on the difference between translation surfaces and half-translation surfaces with respect to the stable norm. Formally, a translation surface is a pair $(S,\omega)$ where $S$ is a Riemann surface and $\omega$ is a holomorphic $1$-form on $S$. The surface is equipped with the flat metric $| \omega |$, with conical singularities located at the zeroes of $\omega$. It is known (due to Masur) that the number of embedded flat cylinders of length $x$ in $S$ grows quadratically as $x$ goes to infinity. Another way to define such a cylinder is as a maximal set of parallel closed geodesics of length less than $x$. If $\gamma$ is a geodesic in a cylinder, it represents a non-trivial homology class. Indeed, since the length of $\gamma$ is 
 $$l(\gamma) = \int_\gamma |\omega| > 0$$
 in particular the pairing $\langle [\gamma], \omega \rangle = \mathlarger{\int}_\gamma \omega$ is non-zero, so $[\gamma]$ is non-zero in $H_1(X,\RR)$. One can show that geodesic in a cylinder are minimizing in their homology class: thus, when counting simple homology classes one can obtain a (quadratic) lower bound by counting cylinders, and we get a quadratic estimate for simple classes on translation surfaces. \\The situation is different on half-translation surfaces. Formally, a half-translation surface is a pair $(S,q)$ where $q$ is a holomorphic quadratic differential, i.e a section of the symmetric square of the holomorphic cotangent bundle of $S$. Again, the flat metric on $S$ is given by $|q|$, and the conical singularities correspond to the zeroes of $q$. Again, one can count cylinders of parallel closed geodesics on $S$ of length less than $x$ and show that this number grows quadratically as $x$ goes to infinity. However, this time if $\gamma$ is closed geodesic in a cylinder it may happen that it represents the trivial homology class. Indeed, since $q$ is not a differential form, the formula $l(\gamma) = \mathlarger{\int}_\gamma |q| > 0$ does not say anything about the pairing $ \eta \in H^1(S,\RR) \mapsto \langle [\gamma], \eta \rangle$ and the homology class $[\gamma]$. For instance, on our surface $\Sigma$ any closed geodesic looping around the central cylinder has trivial homology. Hence on half-translation surfaces there are fewer simple homology classes than there are cylinders of parallel geodesics. 
 Thus, because of this difference between translation surfaces and half-translation surfaces, it makes sense to find a sub-quadratic estimate on the number of simple closed geodesics on a half-translation surface such as $\Sigma$.

\bibliographystyle{alpha}

\end{document}